\documentclass[12pt,leqno]{amsart}
\usepackage{amssymb}
\input amssym.def
\usepackage{amsmath,amsfonts,hyperref,xcolor,mathtools}
\usepackage{amscd}
\usepackage[mathscr]{eucal}
\usepackage{enumitem,orcidlink}

\hypersetup{colorlinks=true,citecolor={purple},linkcolor={teal},urlcolor={violet}}

\newcommand{\N}{{\mathbb N}}
\newcommand{\Z}{{\mathbb Z}}
\newcommand{\Q}{{\mathbb Q}}
\newcommand{\C}{{\mathbb C}}
\newcommand{\R}{{\mathbb R}}

\newtheorem{thm}{Theorem}
\newtheorem{lem}{Lemma}
\newtheorem{cor}{Corollary}
\newtheorem{prop}{Proposition}

\newtheorem{rmk}{Remark}

\newtheorem{ex}{Example}

\newcommand{\thmref}[1]{Theorem~\ref{#1}}
\newcommand{\propref}[1]{Proposition~\ref{#1}}
\newcommand{\lemref}[1]{Lemma~\ref{#1}}
\newcommand{\corref}[1]{Corollary~\ref{#1}}
\newcommand{\rmkref}[1]{Remark~\ref{#1}}

\newcommand{\exref}[1]{Example~\ref{#1}}

\makeatletter
\@namedef{subjclassname@2020}{%
\textup{2020} Mathematics Subject Classification}
\makeatother

\begin{document}


\title[Laurent type expansion of multiple polylogarithms at integer points]
{Laurent type expansion of multiple polylogarithms at integer points}

\author{Pawan Singh Mehta \orcidlink{0009-0001-6717-9063} and Biswajyoti Saha \orcidlink{0009-0009-2904-4860}}

\address{Pawan Singh Mehta and Biswajyoti Saha\\ \newline
Department of Mathematics, Indian Institute of Technology Delhi, 
Hauz Khas, New Delhi 110016, India.}
\email{maz218521@maths.iitd.ac.in, biswajyoti@maths.iitd.ac.in}



\begin{abstract}
In this article, we study the local behaviour of the multiple polylogarithm functions at integer points,
in the $s$-aspect. This is done by writing a Laurent type expansion at integer points, involving certain
power series and rational functions. The coefficients of these power series are the regularised values
of the multiple polylogarithm functions at certain related integer points.
\end{abstract}

\maketitle

\section{Introduction}\label{sec-intro}

Throughout this article, we denote the set of all non-negative integers by $\N$.
For positive integers $r$, $a_1, \ldots, a_r$ and complex numbers $z_1,\ldots, z_r$ with $|z_i|<1$ for all $1\leq i\leq r$,
the multiple polylogarithm (of depth $r$) is defined by the convergent power series
 $$
 \mathrm{Li}(a_1,\ldots,a_r;z_1,\ldots,z_r):=\sum_{n_1>\cdots>n_r>0}\frac{z_1^{n_1}\cdots z_r^{n_r}}{n_1^{a_1}\cdots n_r^{a_r}}.
 $$
This definition can be extended for $|z_i|\leq 1$ for all $1\leq i\leq r$, if $a_1\geq2$.
Further, this series can be considered as a multiple Dirichlet series by fixing complex numbers
$z_1,\ldots, z_r$ and replacing $a_1,\ldots,a_r$ by complex variables $s_1,\ldots,s_r$ such that $(s_1, \ldots, s_r) \in U_r$, i.e.,
\begin{equation}\label{mp}
\mathrm{Li}_{\bf z}(s_1,\ldots,s_r)=\mathrm{Li}_{(z_1,\ldots,z_r)}(s_1,\ldots,s_r):=\sum_{n_1>\cdots>n_r>0}\frac{z_1^{n_1}\cdots z_r^{n_r}}{n_1^{s_1}\cdots n_r^{s_r}},
\end{equation}
where
$$
U_r := \{ (s_1, \ldots, s_r) \in \C^r ~:~ \Re(s_1 + \cdots + s_i) > i
~\text{ for all }~ 1 \le i  \le r \}.
$$
The above series converges normally on compact subsets of $U_r$ and 
therefore defines a holomorphic function there (see \cite[Proposition 2]{BS1}).
The depth $0$ multiple polylogarithm function is taken to be the constant function $1$.
In our recent work \cite{PSBS},
we have shown that the series in \eqref{mp} converges (conditionally) on a larger open domain, depending on ${\bf z}=(z_1,\ldots,z_r)$,
by treating it as
\begin{equation}\label{lim-mp}
\lim_{N \to \infty}\sum_{N>n_1>\cdots>n_r>0}\frac{z_1^{n_1}\cdots z_r^{n_r}}{n_1^{s_1}\cdots n_r^{s_r}}.
\end{equation}
For an integer $r\geq 1$ and ${\bf z}=(z_1,\ldots,z_r) \in \C^r$, where $|z_i|\leq1$ for all  $1\leq i\leq r$, let
$q({\bf z})$ be the smallest positive integer such that $1\leq q({\bf z})\leq r$ and  $z_{q({\bf z})}\neq1$.
If no such $q({\bf z})$ exists, then we set $q({\bf z}):=r+1$. Consider the open set $U_r(\mathbf{z})$ of $\mathbb{C}^r$ defined as follows:
\begin{equation}\label{Urz}
\begin{split}
U_r(\mathbf{z}):=\{ (s_1,\ldots,s_r)\in\mathbb{C}^r: & \ \Re(s_1+\cdots+s_i)>i \text{ if } 1\leq i < q({\bf z}) \\
&\text{ and } \Re(s_1+\cdots+s_i)>i-1 \text{ if } q({\bf z})\leq i \leq r \}.
\end{split}
\end{equation}
We proved the following in \cite[Theorem 2]{PSBS}.

\begin{thm}\label{big-dom}
The multiple Dirichlet series in \eqref{mp} converges on $U_r(\mathbf{z})$.
Moreover, the series in \eqref{mp} defines a holomorphic function on $U_r(\mathbf{z})$.
\end{thm}

This theorem therefore extends the classical result for the depth $1$ Dirichlet series $\mathrm{Li}_{(z)}(s)$.
However, this open domain of convergence is not optimal in general.
For example, as $N \to \infty$, one has
\begin{equation}\label{ex-0}
\sum_{N>n_1>n_2>0}\frac{(-1)^{n_2}}{n_1^{2}n_2^{-1}}=\frac{\log 2}{2}-\frac{\pi^2}{16}+o(1).
\end{equation}
Clearly, $(2,-1) \notin U_2(1,-1)$. In fact, one can see that for $s \in \C$ with $\Re(s)>1$, the limit
$$
\lim_{N \to \infty} \sum_{N>n_1>n_2>0}\frac{(-1)^{n_2}}{n_1^{s}n_2^{-1}}
$$
exists. In \cite[Theorem 3]{PSBS}, we established a more accurate and larger open domain
where the series in \eqref{mp} (viewed as the limit in \eqref{lim-mp}) converges at integer points.
To recall the statement, we need a few more notations. Let $r\geq 1$ be an integer and $\mathbf{z}=(z_1,\ldots,z_r) \in \C^r$
be such that $z_1,\ldots,z_r$ are roots of unity. For each $1 \le i \le r$, define
\begin{equation}\label{I-Q}
J_i(\mathbf{z}):=\{j : 1 \le j \le i \text{ and } z_{[j,i]}=1\}, \ \text{ and } \ Q_i(\mathbf{z}):=|J_i(\mathbf{z})|, 
\end{equation}
the number of elements in $J_i(\mathbf{z})$, where $z_{[j,i]}:=\prod_{k=j}^{i}z_k$ for $1 \le j \le i \le r$. Consider the open set
\begin{equation}\label{Vrz}
V_r(\mathbf{z}):=\{(s_1,\ldots,s_r)\in \C^r : \Re(s_1+\cdots+s_i)>Q_i(\mathbf{z}) \text{ for all } 1\leq i\leq r\}.
\end{equation}
It is easy to see that $V_r(\mathbf{z})=U_r$ if and only if $\mathbf{z}=(1,\ldots,1)$. Moreover,
$U_r(\mathbf{z}) \subseteq V_r(\mathbf{z})$.
We proved the following in \cite[Theorem 3]{PSBS}.

\begin{thm}\label{ad-dom}
Let $r \ge 1$ be an integer and $z_1,\ldots,z_r \in \C$ be roots of unity. The multiple Dirichlet series  in \eqref{mp}
converges at integer points of $V_r(\mathbf{z})$. Moreover, for an integer point $(s_1,\ldots,s_r) \in V_r(\mathbf{z})$
and $k_1,\ldots, k_r\in \mathbb{N}$, the series
\begin{equation}\label{mp-der}
\sum_{n_1>\cdots>n_r>0}\frac{z_1^{n_1}(\log n_1)^{k_1}\cdots z_r^{n_r}(\log n_r)^{k_r}}{n_1^{s_1}\cdots n_r^{s_r}}
\end{equation}
converges when we consider it as the limit
\begin{equation}\label{lim-mp-der}
\lim_{N \to \infty}\sum_{N>n_1>\cdots>n_r>0}\frac{z_1^{n_1}(\log n_1)^{k_1}\cdots z_r^{n_r}(\log n_r)^{k_r}}{n_1^{s_1}\cdots n_r^{s_r}}.
\end{equation}
\end{thm}

This theorem was proved by devising a regularisation process for the series defining the multiple
polylogarithm functions (see \cite[\S 2]{PSBS}). However, the question of holomorphicity of
the multiple polylogarithm functions at integer points of $V_r(\mathbf{z})$ is a delicate one and
was not addressed in \cite{PSBS}.

In this article, which is a follow-up of \cite{PSBS}, we use the regularisation process devised in
our article \cite{PSBS} to address this
problem by writing explicit Laurent type expansion for the multiple polylogarithm functions around
all the integer points in $\C^r$, for any positive integer $r$. A prototype of the desired Laurent type expansion is
$$
\zeta(s)=\frac{1}{s-1}+\sum_{k \ge 0} \frac{(-1)^k \gamma_k}{k!} (s-1)^k,
$$
where for $k \ge 0$, $\gamma_k$ is the Stieltjes constant defined by
$$
\gamma_k:= \lim_{N \to \infty} \left( \sum_{N>n>0} \frac{\log^{k}n}{n} - \frac{\log^{k+1}N}{k+1} \right).
$$
In \cite{BS2}, the second author provided explicit Laurent type expansion for the multiple zeta functions around
integer points. Our work here therefore extends the work of \cite{BS2} to the case of the multiple polylogarithm
functions. This requires us to study the regularised multiple polylogarithm functions, which we define as follows:

For a positive integer $r$, consider ${\bf a}=(a_1,\ldots,a_r) \in \Z^r$ and ${\bf z}=(z_1, \ldots, z_r) \in \C^r$, such that
$z_1, \ldots, z_r$ are roots of unity. Then we consider the formal power series
\begin{align}\label{Reg-Li}
\sum_{k_1,\ldots,k_r\geq 0}\frac{(-1)^{k_1+\cdots+k_r}}{k_1!\cdots k_r!}\ell_{[k_1,\ldots,k_r]}^{(a_1,\ldots,a_r)}{(\mathbf{z})}(s_1-a_1)^{k_1}\cdots(s_r-a_r)^{k_r},
\end{align}
where the constant $\ell_{[k_1,\ldots,k_r]}^{(a_1,\ldots,a_r)}(\mathbf{z})$ is the regularised value of the sequence of partial sums
\begin{align}\label{Reg-series}
\sum_{N>n_1>\cdots>n_r>0} \frac{z_1^{n_1}(\log n_1)^{k_1}\cdots z_r^{n_r}(\log n_r)^{k_r}}{n_1^{a_1}\cdots n_r^{a_r}},
\end{align}
as defined in \cite[Definitions 2, 3]{PSBS}. For the sake of completeness, we recall from \cite{PSBS} that as $N \to \infty$,
the series \eqref{Reg-series} can be expressed as a sum of four different types of terms, namely, one which goes to $\infty$ in
modulus, one which is bounded but oscillates, one which remains constant and one which goes to $0$. The constant term is then
defined to be the regularised value $\ell_{[k_1,\ldots,k_r]}^{(a_1,\ldots,a_r)}(\mathbf{z})$ of the series \eqref{Reg-series}.
For example,
$$
\sum_{N>n>0}(-1)^n=-\frac{(-1)^N}{2} - \frac{1}{2}, \ \sum_{N>n>0}(-1)^n n=-\frac{(-1)^N N}{2} +\frac{(-1)^N}{4} - \frac{1}{4},
$$
and further, as $N \to \infty$, we also have
$$
\sum_{N>n_1>n_2>0}\frac{(-1)^{n_2}n_2}{n_1}= - \frac{\log N}{4} +  \frac{(-1)^N}{4} + \frac{1-\log 2-\gamma}{4} + o(1),
$$
where $\gamma$ is the Euler's constant. Thus, $\ell_{[0]}^{(0)}(-1)=-1/2$, $\ell_{[0]}^{(-1)}(-1)=-1/4$ and
$\ell_{[0,0]}^{(1,-1)}(1,-1)=(1-\log 2-\gamma)/4$. For more details on the above regularisation process we suggest the reader to see \cite[\S2]{PSBS}.

Now for the formal power series \eqref{Reg-Li}, we prove that it
converges in a neighbourhood of  ${\bf a}$ and it can be extended to a meromorphic function on $\C^r$.
We denote this meromorphic function by
$\mathrm{Li}^{\text{Reg}}_{(\mathbf{z};\mathbf{a})}(s_1,\ldots,s_r)$ and call it
the (depth $r$) regularised multiple polylogarithm function for $\mathbf{z}$ and $\mathbf{a}$.
The depth $0$ regularised multiple polylogarithm function is taken to be the constant function $1$. 
If ${\bf a}\in \Z^r\cap U_r(\mathbf{z})$, then as meromorphic functions on $\C^r$ we have
$$
\mathrm{Li}_{\mathbf{z}}(s_1,\ldots,s_r)=\mathrm{Li}^{\text{Reg}}_{(\mathbf{z};\mathbf{a})}(s_1,\ldots,s_r).
$$
But for other integer points of $\C^r$, this is not necessarily true.
For example, if $r=2$ and $\mathbf{z}=(1,-1)$, then around the point $\mathbf{a}=(1,1)$, we have the
following equality between meromorphic functions:
\begin{align}\label{exm-mp-1}
\mathrm{Li}_{\mathbf{z}}(s_1,s_2)=\frac{\mathrm{Li}^{\text{Reg}}_{(-1;1)}{(s_2)}}{s_1-1}+\mathrm{Li}^{\text{Reg}}_{(\mathbf{z};\mathbf{a})}(s_1,s_2).
\end{align}
Hence $s_1=1$ is the only polar hyperplane of $\mathrm{Li}_{\mathbf{z}}(s_1,s_2)$
which passes through the point $(1,1)$.

The formulation of our theorem is more delicate for general integer points, particularly for those lying outside
$\overline{U_r(\mathbf{z})}$. Hence we first consider the case of integer points in $\overline{U_r(\mathbf{z})}$.
This has been stated in \thmref{main-thm-section-2} in \S \ref{behavior-at-boundary} and the proof is given in \S
\ref{proof-boundary}. The intermediate steps involved in the proof of \thmref{main-thm-section-2} are discussed in \S
\ref{sec-inter} and the required proofs are given in \S  \ref{proof-lemmas}.
In \S  \ref{sec-asymp}, we discuss the required setup of asymptotic expansions of sequences of germs of holomorphic and
meromorphic functions with variable coefficients, which we use to prove our theorem for the general integer points in $\C^r$
(see \thmref{thm-gen-point} below). This theorem and its proof appear in \S  \ref{sec-gen}. \thmref{thm-gen-point} also allows us
to discuss the question of holomorphicity of the multiple polylogarithm function at integer points of $V_r(\mathbf{z})$,
which we have mentioned above (see \corref{cor-vrz} below).


\section{Behaviour of the multiple polylogarithms at the integer points of $\overline{U_r(\mathbf{z})}$}\label{behavior-at-boundary}
Let $r\ge 1$ be an integer.
In this section, for $\mathbf{z}=(z_1,\ldots,z_r)\in \C^r$, where $z_i$'s are fixed roots of unity for all $1\leq i\leq r$,
we discuss the local behaviour of the multiple polylogarithms at the integer points of $\overline{U_r(\mathbf{z})}$.
Let $\mathbf{a}=(a_1,\ldots,a_r)\in \Z^r\cap \overline{U_r(\mathbf{z})}$. Then from the definition of $U_r(\mathbf{z})$
(see \eqref{Urz}), we can deduce that the only possible polar hyperplanes of the multiple polylogarithm function
$\mathrm{Li}_{\bf z}(s_1,\ldots,s_r)$ passing through a point $\mathbf{a}$ of $\overline{U_r(\mathbf{z})}$
are of the form $s_1+\cdots+s_i=i$ or $s_1+\cdots+s_i=i-1$ for some $1\leq i\leq r$, depending on the complex 
tuple $\mathbf{z}$ and the point $\mathbf{a}$. To understand this in detail, we recall the following theorem
(see \cite[Lemma 1.2]{FKMT} and \cite[Theorem 9]{BS1}).

\begin{thm}\label{poles-malf}
The function $\mathrm{Li}_{\bf z}(s_1,\ldots,s_r)$
extends to a meromorphic function on the whole of $\C^r$.
If $z_{[1,i]} \neq 1$ for all $1 \le i \le r$, then
$\mathrm{Li}_{\bf z}(s_1,\ldots,s_r)$ is holomorphic on $\C^r$.
Otherwise, if $i_1 < \cdots < i_m$ denote all the indices such that $z_{[1,i_j]}= 1$
for all $1 \le j \le m $, then the set of all possible polar hyperplanes of
$\mathrm{Li}_{\bf z}(s_1,\ldots,s_r)$ can be given as follows:

$(a)$ If $i_1=1$, then $\mathrm{Li}_{\bf z}(s_1,\ldots,s_r)$ is
holomorphic outside the union of the hyperplanes given by the equations
$$
s_1=1; ~s_1 + \cdots+ s_{i_j}=n \ \text{ for all integer } n \le j \ \text{and }
2 \le j \le m.
$$

$(b)$ If $i_1 \not =1$, then $\mathrm{Li}_{\bf z}(s_1,\ldots,s_r)$ is
holomorphic outside the union of the hyperplanes given by the equations
$$
s_1 + \cdots+ s_{i_j}=n \text{ for all integer } n \le j \text{ and }
1 \le j \le m.
$$
\end{thm}

From the above theorem, we can deduce that the plane $s_1+\cdots+s_i=i$ is a
polar hyperplane of $\mathrm{Li}_{\bf z}(s_1,\ldots,s_r)$ passing through the point
$\mathbf{a} \in \Z^r\cap \overline{U_r(\mathbf{z})}$ only if $z_{j}= 1$  for all $1\leq j\leq i$ and $a_1+\cdots+a_i=i$. 
Now suppose that $z_j \ne 1$ for some $1\leq j\leq i$. Then the necessary condition for the plane
$s_1+\cdots+s_i=i-1$ to be a polar hyperplane of $\mathrm{Li}_{\bf z}(s_1,\ldots,s_r)$
passing through the point $\mathbf{a}$ is $z_{[1,i]}=1$ and $a_1+\cdots+a_i=i-1$. 
Note that this can happen only if there exists exactly one integer $i_0$ such that
$1 \le i_0 < i$ and $z_{[1,i_0]} \ne 1$ and $z_{[1,i]}=1$. Hence $z_j=1$ for all $1\le j \le i$,
except for $j=i_0,i_0+1$. This can also be equivalently written as the condition that there exists exactly
one integer $j_0$ such that $1 < j_0 \le i$ and $z_{[j_0,i]} \ne 1$ and $z_{[1,i]}=1$. In fact,
$j_0=i_0+1$.

This prompts us to define an indexing set $\mathcal{I}(\mathbf{z},\mathbf{a})$
which is the collection of all integers $0\leq i\leq r$ and satisfies the following properties:
\begin{enumerate}
\item[(I)] $z_{[1,i]}=1$, \label{cond. 1}
\item[(II)] there is at most one integer $t_i$ such that $1\leq t_i\leq i$ and $z_{[t_i,i]}\neq 1$,\label{cond. 2}
\item[(III)] $a_1+\cdots+a_i=i-a(\mathbf{z})_i$, where $a(\mathbf{z})_i:=$ $\begin{cases}0 &\text{ if } z_{j}=1 \text{ for all } 1\leq j\leq i,\\
1 &\text{ otherwise}. \end{cases}$
\end{enumerate}
By convention, we set $a(\mathbf{z})_0=0$ and hence $0\in \mathcal{I}(\mathbf{z},\mathbf{a})$. Therefore,
$\mathcal{I}(\mathbf{z},\mathbf{a})$ is non-empty for all $\mathbf{z}$ and $\mathbf{a}$.
So for example, if $r=2$ and ${\bf a}=(1,1)$, then $\mathcal{I}(\mathbf{z},\mathbf{a})=\{0,1,2\}$ for ${\bf z}=(1,1)$,
$\mathcal{I}(\mathbf{z},\mathbf{a})=\{0,1\}$ for ${\bf z}=(1,-1)$, and $\mathcal{I}(\mathbf{z},\mathbf{a})=\{0\}$ for both
${\bf z}=(-1,-1)$ and ${\bf z}=(-1,1)$.

\begin{rmk}\label{def-i_0}\rm
In condition (II) of the above definition, if for some $i\in \mathcal{I}(\mathbf{z},\mathbf{a})$, there is no $t_i$
such that $1\leq t_i\leq i$ and $z_{[t_i,i]}\neq 1$, then we take $t_i=0$. 
\end{rmk}

For $i \ge 1$ in $\mathcal{I}(\mathbf{z},\mathbf{a})$, we define the function $H_{j}^{(i)}(s_j,\ldots,s_i ; z_{[t_i,i]})$
in $s_j,\ldots,s_i$, for each $1\leq j\leq i$, as follows:
\begin{align}\label{denominator}
H_{j}^{(i)}(s_j,\ldots,s_i ; z_{[t_i,i]}):=
\begin{cases}
 s_j+\cdots+s_i-(i-j) &\text{ if } 1\leq j<t_i,\\
  1-z_{[t_i,i]} & \text{ if } j=t_i,\\
  s_j+\cdots+s_i-(i-j+1) &\text{ if } t_i< j\leq i.
 \end{cases}
\end{align}


With these notations in place, we are ready to state our main theorem of this section.

\begin{thm}\label{main-thm-section-2}
Let $r\geq 1$ be an integer and $\mathbf{z}=(z_1,\ldots,z_r)\in \C^r$ be such that $z_i$'s are roots of unity for all $1\leq i\leq r$.
Let $\mathbf{a}=(a_1,\ldots,a_r)\in \Z^r\cap \overline{U_r(\mathbf{z})}$. Then the power series 
\begin{align}\label{power-series}
\sum_{k_1,\ldots,k_r\geq 0}\frac{(-1)^{k_1+\cdots+k_r}}{k_1!\cdots k_r!}\ell_{[k_1,\ldots,k_r]}^{(a_1,\ldots,a_r)}(\mathbf{z})(s_1-a_1)^{k_1}\cdots(s_r-a_r)^{k_r}
\end{align}
converges in a neighbourhood of ${\bf a}$ and can be extended to a meromorphic function on $\C^r$. 
Moreover, if this meromorphic function is denoted by $\mathrm{Li}^{\rm{Reg}}_{(\mathbf{z};\mathbf{a})}(s_1,\ldots,s_r)$, 
then we have the following equality between the meromorphic functions:
\begin{align}\label{behaviour-boundary}
\mathrm{Li}^{\rm{Reg}}_{(\mathbf{z};\mathbf{a})}(s_1,\ldots,s_r)=
\sum_{i\in \mathcal{I}(\mathbf{z},\mathbf{a})}(-1)^{i}
\frac{\mathrm{Li}_{(z_{i+1},\ldots,z_r)}(s_{i+1},\ldots,s_r)}{\prod_{1\leq j\leq i}H_{j}^{(i)}(s_j,\ldots,s_i;z_{[t_i,i]})}.
\end{align}
\end{thm}

For an integer point $\mathbf{a}=(a_1,\ldots,a_r)\in \Z^r\cap \overline{U_r(\mathbf{z})}$, in order to get a Laurent type
expansion of $\mathrm{Li}_{\bf z}(s_1,\ldots,s_r)$ around $\mathbf{a}$ from \eqref{behaviour-boundary}, we need an
{\it inversion process}. For this we first note that for every $i \in \mathcal{I}(\mathbf{z},\mathbf{a})$, the point
$(a_{i+1},\ldots,a_r)\in \Z^{r-i}\cap \overline{U_{r-i}(z_{i+1},\ldots,z_r)}$. This is obvious for $i=0$. For $i \ge 1$ and 
$i \in \mathcal{I}(\mathbf{z},\mathbf{a})$, we have $a_1+\cdots+a_i=i-a({\bf z})_i$.
Let $k$ be an integer such that $i\leq k\leq r$. Note that by definition, we have
$$
a_1+\cdots+a_k=i-a({\bf z})_i+a_{i+1}+\cdots+a_k\geq k-a({\bf z})_k.
$$
Therefore, $a_{i+1}+\cdots+a_k\geq (k-i)-\left(a({\bf z})_k-a({\bf z})_i\right)$. Note that $0\leq a({\bf z})_k-a({\bf z})_i\leq 1$,
as $a({\bf z})_i =1$ implies  $a({\bf z})_k =1$. Now we only need to study the case $a({\bf z})_k-a({\bf z})_i=1$, i.e.,
$a({\bf z})_i=0$ and $a({\bf z})_k=1$. This happens if and only if $z_1=\cdots=z_i=1$ and there exists an integer $j$
such that $i<j\leq k$ and $z_j\neq 1$. Thus, $a({\bf z})_k= a(z_{i+1},\ldots,z_k)_{k-i}=1$ and in this case
$$
a_{i+1}+\cdots+a_k\geq (k-i)-a(z_{i+1},\ldots,z_k)_{k-i}.
$$
Therefore, we get $(a_{i+1},\ldots,a_r)\in \Z^{r-i}\cap{\overline {U_{r-i}(z_{i+1},\ldots, z_r})}$. 

For $\mathbf{a}=(a_1,\ldots,a_r)\in \Z^r\cap \overline{U_r(\mathbf{z})}$, we write
\eqref{behaviour-boundary} for every point of the form $(a_{i+1},\ldots,a_r)$, where
$i \in \mathcal{I}(\mathbf{z},\mathbf{a})$. All these equations can be concisely written
in the following matrix form:
$$
{\bf V}^{\text{Reg}}={\bf M}{\bf V},
$$
where ${\bf V}$ and ${\bf V}^{\text{Reg}}$ are column vectors indexed by the elements of 
$\mathcal{I}(\mathbf{z},\mathbf{a})$, such that for each $i\in \mathcal{I}(\mathbf{z},\mathbf{a})$, 
the corresponding entries of ${\bf V}$ and ${\bf V}^{\text{Reg}}$ are $\mathrm{Li}_{(z_{i+1},\ldots,z_r)}(s_{i+1},\ldots,s_r)$
and $\mathrm{Li}^{\text{Reg}}_{(z_{i+1},\ldots,z_r;a_{i+1},\ldots,a_r)}(s_{i+1},\ldots,s_r)$, respectively,
and ${\bf M}$ is an upper triangular matrix with entries in the field of rational functions
$\C(s_1,\ldots,s_r)$, with the diagonal entries as $1$. Hence ${\bf M}$ is an invertible matrix with entries in
$\C(s_1,\ldots,s_r)$ and we can write
$$
{\bf V}={\bf M}^{-1}{\bf V}^{\text{Reg}}.
$$
Now the first entry of the above matrix equation gives an expression of the form
$$
\mathrm{Li}_{\bf z}(s_{1},\ldots,s_r)=\sum_{i\in \mathcal{I}(\mathbf{z},\mathbf{a})}R_i(s_1,\ldots,s_i)
\mathrm{Li}^{\text{Reg}}_{(z_{i+1},\ldots,z_r;a_{i+1},\ldots,a_r)}(s_{i+1},\ldots,s_r),
$$
where $R_0(\varnothing):=1$ and for $i \in \mathcal{I}({\bf z,a})$ with $i \ge 1$, $R_i(s_1,\ldots,s_i)$
is a rational function in $s_1,\ldots,s_i$. This is the desired Laurent type expansion of
$\mathrm{Li}_{\bf z}(s_{1},\ldots,s_r)$ around ${\bf a}$.
This can be seen in the following examples.
\begin{ex}\rm
For $r=2, \mathbf{z}=(1,-1)$ and $\mathbf{a}=(1,1)$, we have seen that
$\mathcal{I}(\mathbf{z},\mathbf{a})=\{0,1\}$ and hence by \eqref{behaviour-boundary},  we have
$$
\mathrm{Li}^{\text{Reg}}_{(\mathbf{z};\mathbf{a})}(s_1,s_2)=
\mathrm{Li}_{\mathbf{z}}(s_1,s_2)-\frac{\mathrm{Li}_{(-1)}{(s_2)}}{s_1-1}.
$$
Again, around the point $s_2=1$, we have $\mathrm{Li}^{\text{Reg}}_{(-1;1)}{(s_2)}=\mathrm{Li}_{(-1)}{(s_2)}$, and thus
we recover \eqref{exm-mp-1}.
\end{ex}
A more involved example is as follows:
\begin{ex}\rm
For $r=3, {\bf z}=(1,-1,-1)$ and ${\bf a}=(1,1,0)$, we have $\mathcal{I}(\mathbf{z},\mathbf{a})=\{0,1,3\}$. In this case, we have
$$
\begin{bmatrix}
 \mathrm{Li}^{\text{Reg}}_{({\bf z}; {\bf a})}(s_1,s_2,s_3)\\
 \mathrm{Li}^{\text{Reg}}_{(-1,-1;1,0)}(s_2,s_3)\\
1
\end{bmatrix}
= 
\begin{bmatrix}
1  &-\frac{1}{s_1-1} &-\frac{1}{2(s_3+s_2-1)(s_3+s_2+s_1-2)} \\ 
     0                         &1             &\frac{1}{2(s_3+s_2-1)}\\
         0                            &0                   &1
\end{bmatrix}
 \begin{bmatrix}
 \mathrm{Li}_{\bf z}(s_1,s_2,s_3)\\
 \mathrm{Li}_{(-1,-1)}(s_2,s_3)\\
1
\end{bmatrix}.
$$
Therefore,
$$ 
 \begin{bmatrix}
 \mathrm{Li}_{\bf z}(s_1,s_2,s_3)\\
 \mathrm{Li}_{(-1,-1)}(s_2,s_3)\\
1
\end{bmatrix}
=
\begin{bmatrix}
1  &\frac{1}{s_1-1} &-\frac{1}{2(s_1-1)(s_3+s_2+s_1-2)} \\ 
     0                         &1             &-\frac{1}{2(s_3+s_2-1)}\\
         0                            &0                   &1
\end{bmatrix}
\begin{bmatrix}
 \mathrm{Li}^{\text{Reg}}_{({\bf z}; {\bf a})}(s_1,s_2,s_3)\\
 \mathrm{Li}^{\text{Reg}}_{(-1,-1;1,0)}(s_2,s_3)\\
1
\end{bmatrix},
$$
and hence we get the Laurent type expansion around the point ${\bf a}=(1,1,0)$, given by
$$
\mathrm{Li}_{\bf z}(s_1,s_2,s_3)= \mathrm{Li}^{\text{Reg}}_{({\bf z}; {\bf a})}(s_1,s_2,s_3)+\frac{ \mathrm{Li}^{\text{Reg}}_{(-1,-1;1,0)}(s_2,s_3)}{s_1-1}-\frac{1}{2(s_1-1)(s_3+s_2+s_1-2)}.
$$
\end{ex}

In the next section, we setup the intermediate steps required to prove \thmref{main-thm-section-2}.

\section{Intermediate steps for \thmref{main-thm-section-2}}\label{sec-inter}


We first recall the definition of the multiple polylogarithm-star function. Let $r\geq 1$ be an integer. In what follows,
${\bf z}=(z_1,\ldots,z_r) \in \C^r$ is fixed such that $|z_i| \le1$ for all $1\leq i\leq r$, unless otherwise stated. The multiple polylogarithm-star
function of depth $r$ is defined on $U_r$ by
\begin{align}\label{mp-star}
\mathrm{Li}^{\star}_{\bf z}(s_1,\ldots,s_r)=\mathrm{Li}^{\star}_{(z_1,\ldots,z_r)}(s_1,\ldots,s_r)
:= \sum_{n_1\geq \cdots \geq n_r \geq 1}\frac{z_1^{n_1}\cdots z_r^{n_r}}{n_1^{s_1}\cdots n_r^{s_r}}.
\end{align}
This series converges uniformly on every compact subset of the open set $U_r$. 
Therefore, it defines a holomorphic function on $U_r$. Since the multiple polylogarithm-star function of
depth $r$ can be written as a linear combination
of the multiple polylogarithm functions of depth $ \le r$ (can be seen using an analogue of \cite[eq. (33)]{MSV}),
it can be extended as a meromorphic  function on $\C^r$. As before, 
the depth $0$ multiple polylogarithm-star function is taken to be the constant function $1$.

%
 
Similarly, for an integer $N\geq1$ and $(s_1,\ldots,s_r)\in U_r$, consider the tail
\begin{align}\label{poly-star-N}
 \mathrm{Li}^{\star}_{\bf z}(s_1,\ldots,s_r)_{\ge N}= \mathrm{Li}^{\star}_{(z_1,\ldots,z_r)}(s_1,\ldots,s_r)_{\geq N} := 
 \sum_{n_1\geq \cdots \geq n_r \geq N}\frac{z_1^{n_1}\cdots z_r^{n_r}}{n_1^{s_1}\cdots n_r^{s_r}}.
\end{align}
Again for each integer $N\geq1$, $\mathrm{Li}^{\star}_{\bf z}(s_1,\ldots,s_r)_{\geq N}$ is holomorphic on 
 $U_r$ and can be extended to a meromorphic function on $\C^r$. As usual, the depth $0$ analogue of this function
 is taken to be the constant function $1$. We first note that the tail of the multiple polylogarithm-star function
 $\mathrm{Li}^{\star}_{\bf z}(s_1,\ldots,s_r)_{\geq N}$ satisfies a translation formula, which is proved in the following section.

\begin{prop}\label{trans}
Let $N \ge 2$ and $r\geq 1$ be integers and $(s_1,\ldots,s_r)\in U_r$. Then for $r=1$, we have
\begin{equation}\label{eq1}
\left(1-\frac{1}{z_1}\right)\mathrm{Li}^{\star}_{(z_1)}(s_1-1)_{\geq N} +  \frac{z_1^{N-1}}{N^{s_1-1}}
=\sum_{k\geq0} (-1)^k\frac{(s_1-1)_{k+1}}{(k+1)!}\mathrm{Li}^{\star}_{(z_1)}(s_1+k)_{\geq N},
\end{equation}
and for $r>1$, we have
 \begin{align}\label{eq2}
\begin{split}
&\left(1-\frac{1}{z_1}\right)\mathrm{Li}^{\star}_{\bf z}(s_1-1,s_{2},\ldots,s_r)_{\geq N}+
\frac{1}{z_1}\mathrm{Li}^{\star}_{(z_{[1,2]},z_{3},\ldots,z_r)}(s_1+s_{2}-1,s_{3},\ldots,s_r)_{\geq N}\\
&=\sum_{k\geq0}(-1)^k\frac{(s_1-1)_{k+1}}{(k+1)!}\mathrm{Li}^{\star}_{\bf z}(s_1+k,s_2,\ldots,s_r)_{\geq N},
\end{split}
\end{align}
where for any complex number $s$ and non-negative integer $k$, $(s)_k$ denotes the Pochhammer symbol 
 $s(s+1)\cdots(s+k-1)$.
\end{prop}

 \begin{rmk}\rm
 In \eqref{eq1} and \eqref{eq2} above, the first term on the left-hand side is considered to be $0$ if $z_1=1$ and when
 $z_1\ne 1$, note that $(s_1-1 ,s_2,\ldots,s_r)\in U_r({\bf z})$ and hence the series defining
$\mathrm{Li}^{\star}_{\bf z}(s_1-1,s_{2},\ldots,s_r)_{\geq N}$ converges (as can be seen from \cite{PSBS}).
The identities \eqref{eq1} and \eqref{eq2} also extend to $\C^r$, as equality between the meromorphic functions.
 \end{rmk}
 
 Now to prove \thmref{main-thm-section-2}, the key step is to find the relation between 
 the multiple polylogarithm functions, the tail of the multiple polylogarithm-star functions
 and the partial sum of the series defining the multiple polylogarithm functions.
 Let $r\geq1$ and $N\geq2$ be integers. Consider the holomorphic function
 \begin{align*}
 \mathrm{Li}_{\bf z}(s_1,\ldots,s_r)_{<N} = \mathrm{Li}_{(z_1,\ldots,z_r)}(s_1,\ldots,s_r)_{<N}:=\sum_{N>n_1>\cdots>n_r>0}\frac{z_1^{n_1}\cdots z_r^{n_r}}{n_1^{s_1}\cdots n_r^{s_r}}
 \end{align*}
 on $\C^r$. Then we have the following analogue of \cite[Theorem 4]{BS2}.
 
 \begin{thm} \label{combi}
The following equality holds between the meromorphic functions on $\C^r$:
\begin{align}\label{combi-indenty}
\mathrm{Li}_{\bf z}(s_1,\ldots,s_r)_{< N}=\sum_{i=0}^{r}(-1)^i\mathrm{Li}^{\star}_{(z_i,\ldots,z_1)}(s_i,\ldots,s_1)_{\geq N}\mathrm{Li}_{(z_{i+1},\ldots,z_r)}(s_{i+1},\ldots,s_r).
\end{align}
\end{thm}

\begin{proof}
The proof can be deduced following the proof of \cite[Theorem 4]{BS2}, as it uses only the combinatorics between the indices.
\end{proof}

In view of the above combinatorial formula we need to study 
$\mathrm{Li}^{\star}_{(z_r,\ldots,z_1)}(s_r,\ldots,s_1)_{\geq N}$ around the integer points in $\overline{U_r(\mathbf{z})}$
for $r \ge 1$. Several cases pertaining to this are covered in the following lemmas, whose proofs require \propref{trans}.
The proofs of these lemmas are given in the following section.


\begin{lem}\label{boundary-1}
Let $r\geq 1$  be an integer.
Let $\mathbf{a}=(a_1,\ldots,a_r)\in \Z^r$ be such that $a_1+\cdots+a_r>r-a(\mathbf{z})_r$. Then there exists a 
non-zero polynomial $P(s_1,\ldots,s_r)$ (with complex coefficients) and a polydisc $D$ around $\mathbf{a}$
such that the 
function $P(s_1,\ldots,s_r)\mathrm{Li}_{(z_r,\ldots,z_1)}^{\star}(s_r,\ldots,s_1)_{\geq N}$ is holomorphic on $D$
for every integer $N \ge 2$, and as $N\rightarrow\infty$,
$$
\left\| P(s_1,\ldots,s_r)\mathrm{Li}_{(z_r,\ldots,z_1)}^{\star}(s_r,\ldots,s_1)_{\geq N}\right\|_D\rightarrow 0.
$$ 
\end{lem}

\begin{lem}\label{boundary-2}
Let $r\geq 1$  be an integer.
Assume that there exists at most one integer $t_r$ such that $1\leq t_r\leq r$ 
and $z_{[t_r,r]}\neq 1$ or $t_r=0$. Let $\mathbf{a}=(a_1,\ldots,a_r)\in \Z^{r}$ be such that 
$a_1+\cdots+a_r=r-a(\mathbf{z})_r$. Then there exists a 
non-zero polynomial $P(s_1,\ldots,s_r)$ (with complex coefficients) and a polydisc $D$ around $\mathbf{a}$ such that 
the function $P(s_1,\ldots,s_r)\mathrm{Li}_{(z_r,\ldots,z_1)}^{\star}(s_r,\ldots,s_1)_{\geq N}$ is holomorphic on $D$
for every integer $N \ge 2$, and as $N\rightarrow\infty$, 
$$
\left\|P(s_1,\ldots,s_r)\mathrm{Li}_{(z_r,\ldots,z_1)}^{\star}(s_r,\ldots,s_1)_{\geq N}
-\frac{P(s_1,\ldots,s_r)z_{[1,r]}^N N^{-(s_1+\cdots+s_r+a(\mathbf{z})_r-r)}}
{\prod_{1\leq j\leq r}H_{j}^{(r)}(s_j,\ldots,s_r;z_{[t_r,r]})}\right\|_D\rightarrow 0.
$$
\end{lem}

We also need the following lemma.

\begin{lem}\label{boundary-3}
Let $r\geq 2$ be an integer.
Assume that there exist at least two integers $i_1, i_2$ such that $1\leq i_1<i_2 \leq r$ and $z_{[i_1,r]},z_{[i_2,r]}\neq 1$.
Let $\mathbf{a}=(a_1,\ldots,a_r)\in \Z^{r}$ be such that $a_1+\cdots+a_r=r-a(\mathbf{z})_r$. 
Then there exists a non-zero polynomial $P(s_1,\ldots,s_r)$ (with complex coefficients)
and a polydisc $D$ around $\mathbf{a}$ such that the 
function $P(s_1,\ldots,s_r)\mathrm{Li}_{(z_r,\ldots,z_1)}^{\star}(s_r,\ldots,s_1)_{\geq N}$ is holomorphic on $D$
for every integer $N \ge 2$, and as $N\rightarrow\infty$,
$$
\left\| P(s_1,\ldots,s_r)\mathrm{Li}_{(z_r,\ldots,z_1)}^{\star}(s_r,\ldots,s_1)_{\geq N}\right\|_D\rightarrow 0.
$$ 
\end{lem}

Now we consider ${\bf z}=(z_1,\ldots,z_r) \in \C^r$ such that $z_1,\ldots,z_r$ are fixed roots of unity.
Let $\mathcal{C}_{\bf z}$ be the $\C$-algebra
generated by the constant sequence $(1)_{n\geq 1}$ and the sequences of the form $(z_i^n)_{n\geq 1}$ for $1 \le i \le r$
(see \cite[\S 2]{PSBS}). Note that as a $\C$-vector space, ${\mathcal C}_{\bf z}$ is finite dimensional, as the set
$\{z_1^{k_1} \cdots z_r^{k_r}: k_i \in \N \text{ for } 1 \le i \le r\}$ is finite.
We need the following lemma which is also key to the proof of  \thmref{main-thm-section-2}.
 
\begin{lem}\label{Taylor-lem}
Let $r\geq 1$ be an integer and $\mathbf{a}=(a_1,\ldots,a_r)$.
Let $D$ be a polydisc around $\mathbf{a}$
in $\C^r$ and $(u_N(s_1,\ldots,s_r))_{N\geq 1}$ be a sequence of holomorphic functions on $D$ such that it
converges uniformly to a holomorphic function $u(s_1,\ldots,s_r)$ in $D$. Suppose that
for all non-negative integers $k_1,\ldots,k_r$, the ${\bf k}=(k_1,\ldots,k_r)$-th coefficient of the
Taylor expansion of $u_N(s_1,\ldots,s_r)$ around $\mathbf{a}$ is of the form $P_{\bf k}(N,\log N)+o(1)$
as $N\rightarrow\infty$, where $P_{\bf k}(N,\log N)$ is a polynomial of the form
$\sum_{i,j \ge 0}a_{\bf k}^{(i,j)}(N)N^{i}(\log N)^j$ such that $(a_{\bf k}^{(i,j)}(N))_{N\geq1}\in \mathcal{C}_{\bf z}$ 
for all pair $(i,j)$. Then $P_{\bf k}(N,\log N)=a_{\bf k}^{(0,0)}(N)$ and the sequence
$(a_{\bf k}^{(0,0)}(N))_{N\geq 1}$ converges to a complex number $a_{\bf k}$. Moreover, the function
$u(s_1,\ldots,s_r)$ has the following power series expansion around ${\bf a}$ in the polydisc $D$:
$$
\sum_{k_1,\ldots,k_r\geq 0}a_{(k_1,\ldots,k_r)}(s_1-a_1)^{k_1}\cdots(s_r-a_r)^{k_r}.
$$
\end{lem}


\section{Proof of \thmref{main-thm-section-2}}\label{proof-boundary}
Now we are ready to prove \thmref{main-thm-section-2}. We need the following notations.
Suppose $z_{[1,i]}\neq 1$ for some $1\leq i\leq r$. Then clearly $i\notin \mathcal{I}(\mathbf{z},\mathbf{a})$.
Also, in that case $a(\mathbf{z})_i=1$. Now we define the set
$\mathcal{I}{'}(\mathbf{z},\mathbf{a})$ which is the collection of all the indices $1\leq i\leq r$ such that
\begin{enumerate}
\item[(I)] $z_{1}\neq 1$ and $z_{j}=1$ for all  $2\leq j\leq i$,
\item[(II)] $a_1+\cdots+a_i=i-1$.
\end{enumerate}

\begin{rmk}\label{effective-i}\rm
It is clear that $\mathcal{I}(\mathbf{z},\mathbf{a})\cap \mathcal{I}{'}(\mathbf{z},\mathbf{a})=\varnothing$ and the set 
$\mathcal{I}^{}(\mathbf{z},\mathbf{a})\cup\mathcal{I}{'}(\mathbf{z},\mathbf{a})$ is the collection of all the indices $0\leq i\leq r$
such that $a_1+\cdots+a_i=i-a(\mathbf{z})_i$ and there exists at most one integer $1\leq t_i\leq i$ with $z_{[t_i,i]}\neq 1$.
\end{rmk}

\begin{proof}[Proof of \thmref{main-thm-section-2}]
Since ${\bf a} \in \Z^r\cap \overline{U_r(\mathbf{z})}$, we have $(a_1,\ldots,a_i)\in \Z^i\cap \overline{U_i(\mathbf{z})}$
for each $1\leq i\leq r$. 
Hence $a_1+\cdots+a_i\geq i-a({\mathbf{z})_i}$ for $i=1,\ldots,r$. Now applying \lemref{boundary-1}, \ref{boundary-2} and \ref{boundary-3} to each of 
$\mathrm{Li}_{(z_i,\ldots,z_1)}^{\star}(s_i,\ldots,s_1)_{\geq N}$ for each $i=1,\ldots,r$, we can find a common polydisc $D_1$ around ${\bf a}$ 
and a common polynomial $P_1(s_1,\ldots,s_r)$
with the desired property. Furthermore, we can find a common polynomial $P_2(s_1,\ldots,s_r)$ and a polydisc $D_2$  around ${\bf a}$ 
such that the function $P_2(s_1,\ldots,s_r)\mathrm{Li}_{(z_{i+1},\ldots,z_r)}(s_{i+1},\ldots,s_r)$ is holomorphic on $D_2$ for each $i=0,\ldots,r-1$. 
Hence, in view of \thmref{combi}, we consider $D=D_1\cap D_2$ and $P(s_1,\ldots,s_r)=P_1(s_1,\ldots,s_r)P_2(s_1,\ldots,s_r)$
to see (using \rmkref{effective-i}) that the sequence of holomorphic functions $(u_N)_{N\geq 1}$ on $D$, given by
\begin{align*}
u_N:= &P(s_1,\ldots,s_r) \mathrm{Li}_{\bf z}(s_1,\ldots,s_r)_{<N} \\
                &-P(s_1,\ldots,s_r)\sum_{i \in \mathcal{I}(\mathbf{z},\mathbf{a})}\frac{(-1)^{i}N^{-(s_1+\cdots+s_i+a({\mathbf{z}})_i-i)}
                 \mathrm{Li}_{(z_{i+1},\ldots,z_r)}(s_{i+1},\ldots,s_r)}{\prod_{1\leq j\leq i}H_{j}^{(i)}(s_j,\ldots,s_i;z_{[t_i,i]})}\\
                &-P(s_1,\ldots,s_r)\sum_{i \in \mathcal{I}{'}(\mathbf{z},\mathbf{a})}\frac{(-1)^{i} z_{[1,i]}^{N}N^{-(s_1+\cdots+s_i+1-i)}
                \mathrm{Li}_{(z_{i+1},\ldots,z_r)}(s_{i+1},\ldots,s_r)}{\prod_{1\leq j\leq i}H_{j}^{(i)}(s_j,\ldots,s_i;z_{[t_i,i]})},
\end{align*}
converges normally to $0$ on $D$, as $N \to \infty$.

Note that for $i\in \mathcal{I}(\mathbf{z},\mathbf{a})\cup \mathcal{I}{'}(\mathbf{z},\mathbf{a})$,
 the Taylor series expansion of $N^{-(s_1+\cdots+s_i+a({\mathbf{z}})_i-i)}$ around the point $(a_1,\ldots,a_i)$ is 
 $$
 \sum_{k_1,\ldots,k_i\geq 0}\frac{(-\log N)^{k_1+\cdots+k_i}}{k_1!\cdots k_i!}(s_1-a_1)^{k_1}\cdots(s_i-a_i)^{k_i},
 $$
as for such $i$, we have $a_1+\cdots+a_i=i-a({\mathbf{z}})_i$. Therefore, applying \lemref{Taylor-lem} to the sequence of functions
 $(u_N)_{N\geq 1}$,  we get that the formal power series
 $$
 P(s_1,\ldots,s_r)\sum_{k_1,\ldots,k_r\geq 0}\frac{(-1)^{k_1+\cdots+k_r}}{k_1!\cdots k_r!}
 \ell_{[k_1,\ldots,k_r]}^{(a_1,\ldots,a_r)}{(\mathbf{z})}(s_1-a_1)^{k_1}\cdots(s_r-a_r)^{k_r}
 $$
 converges on $D$, and it is equal to 
  $$
 P(s_1,\ldots,s_r)\sum_{i \in \mathcal{I}(\mathbf{z};\mathbf{a})}
(-1)^i \frac{ \mathrm{Li}_{(z_{i+1},\ldots,z_r)}(s_{i+1},\ldots,s_r)}{\prod_{1\leq j\leq i}H_{j}^{(i)}(s_j,\ldots,s_i;z_{[t_i,i]})}
  $$
on $D$, taking only the contributions from the constant terms. Note that the indices
$i\in \mathcal{I}{'}(\mathbf{z},\mathbf{a})$ do not takes part in
this expression because $z_{[1,i]}\neq 1$, for such $i$. This implies that the
formal power series \eqref{power-series} converges on $D$. As it is denoted by
$\mathrm{Li}^{\text{Reg}}_{(\mathbf{z};\mathbf{a})}(s_1,\ldots,s_r)$, therefore in the polydisc $D$, we have
 \begin{align*}
\mathrm{Li}^{\text{Reg}}_{(\mathbf{z};\mathbf{a})}(s_1,\ldots,s_r)=\sum_{i\in \mathcal{I}(\mathbf{z},\mathbf{a})}(-1)^{i}
\frac{\mathrm{Li}_{(z_{i+1},\ldots,z_r)}(s_{i+1},\ldots,s_r)}{\prod_{1\leq j\leq i}H_{j}^{(i)}(s_j,\ldots,s_i;z_{[t_i,i]})}.
\end{align*}
 This completes the proof of \thmref{main-thm-section-2}.
 \end{proof}

\section{Proof of the intermediate steps}\label{proof-lemmas}

Since the proofs of \lemref{boundary-1}, \ref{boundary-2} and \ref{boundary-3}, require \propref{trans},
we first prove the following lemma which is required to prove \propref{trans}.

\begin{lem}\label{uni-convg}
Let $N \ge 2$ and $r\geq1$ be integers. Let $(z_1,\ldots,z_r)\in\C^r$ be such that $|z_i|\leq 1$ for all $1\leq i\leq r$. 
Let $K$ be any compact subset of $\C^r$ and $k_0$ be the smallest non-negative integer such 
that $(s_1+k_0,s_2,\ldots,s_r)\in U_r$ for all $(s_1,\ldots,s_r)\in K$. Then the sequence of functions
$$
\left(\frac{(s_1-1)_{k+1}}{(k+1)!}\frac{z_1^{n_1}z_2^{n_2}\cdots z_r^{n_r}}{n_1^{s_1+k}n_2^{s_2}\cdots n_r^{s_r}}\right)_{n_1\geq \cdots\geq n_r\geq N, k\geq k_0}
$$
converges normally on $K$. Moreover, the sum is $O(N^{-\epsilon})$ as $N\rightarrow\infty$ for some $\epsilon>0$.
\end{lem}

 \begin{rmk}\label{series-tail} \rm
 For any point ${\bf a}=(a_1,\ldots,a_r)\in \C^r$, there exists a non-negative integer $k_0$ such that 
 $(a_1+k_0,a_2,\ldots,a_r)\in U_r$. Therefore, we can find a suitably small polydisc $D$
 around the point $(a_1+k_0,a_2,\ldots,a_r)$ in $U_r$ such that $\overline{D}\subseteq U_r$. Hence,
 there exists a polydisc $D$ around ${\bf a}$ such that for all $k\geq k_0$, 
 the function $\mathrm{Li}^{\star}_{\bf z}(s_1+k,s_2,\ldots,s_r)_{\geq N}$ is holomorphic on $D$.
 Moreover, applying \lemref{uni-convg}, we have
 $$
 \left\|\sum_{k\geq k_0}(-1)^k\frac{(s_1-1)_{k+1}}{(k+1)!}\mathrm{Li}^{\star}_{\bf z}(s_1+k,s_2,\ldots,s_r)_{\geq N}\right\|_D=o(1),
 $$
 as $N\rightarrow\infty$.
 \end{rmk}

\subsection{Proof of \lemref{uni-convg}}
 Let $A:=\sup_{(s_1,\ldots,s_r)\in K}|s_1-1|$, then for any positive integers $n_1,\ldots,n_r$, we have
$$
\left \|\frac{(s_1-1)_{k+1}}{(k+1)!}\frac{z_1^{n_1}z_2^{n_2}\cdots z_r^{n_r}}{n_1^{s_1+k}n_2^{s_2}\cdots n_r^{s_r}}\right \|_K
\leq \frac{A(A+1)\cdots (A+k)}{2^{k-k_0}(k+1)!}
\left \|\frac{1}{n_1^{s_1+k_0}n_2^{s_2}\cdots n_r^{s_r}}\right \|_K.
$$
Since $\sum_{k\geq k_0}\frac{A(A+1)\cdots (A+k)}{2^{k-k_0}(k+1)!}<\infty$, 
it is enough to prove that there exists $\epsilon>0$ such that the sum, 
$\sum_{n_1\geq\cdots\geq n_r\geq N}\left \|\frac{1}{n_1^{s_1}\cdots n_r^{s_r}}\right \|_X
=O(N^{-\epsilon})$ as $N\rightarrow\infty$, for any compact subset $X$ of $U_r$.
Since $X$ is a compact subset of $U_r$, there exists $(\sigma_1,\ldots,\sigma_r)\in\R^r\cap U_r$ 
such that $X\subseteq D(\sigma_1,\ldots,\sigma_r)$, where
$$
D(\sigma_1,\ldots,\sigma_r):=\{(s_1,\ldots,s_r)\in\C^r: \Re(s_i)\geq \sigma_i \text{ for all } 1\leq i\leq r\}.
$$
Then,
$\left \|\frac{1}{n_1^{s_1}\cdots n_r^{s_r}}\right \|_{D(\sigma_1,\ldots,\sigma_r)}\leq
\frac{1}{n_1^{\sigma_1}\cdots n_r^{\sigma_r}}$ 
for any positive integers $n_1,\ldots,n_r$. Therefore,
$$
\sum_{n_1\geq n_2}\frac{1}{n_1^{\sigma_1}}\leq\frac{1}{n_2^{\sigma_1}}+\int_{n_2}^{\infty}x^{-\sigma_1}dx
=\frac{1}{n_2^{\sigma_1}}+\frac{1}{1-\sigma_1}\frac{1}{n_2^{\sigma_1-1}} .
$$
Using this inequality (and its avatars), we can write
$$
\sum_{n_1\geq \cdots\geq n_r\geq N}\frac{1}{n_1^{\sigma_1}\cdots n_r^{\sigma_r}}\leq\sum_{i=0}^{r}
\frac{T_i(\sigma_1,\ldots,\sigma_i)}{N^{\sigma_1+\cdots+\sigma_r-i}},
$$
where for each $1 \le i \le r$, $T_i(\sigma_1,\ldots,\sigma_i)$ is a rational function whose denominator
is made of the factors of the form
$\sigma_1+\cdots+\sigma_j-k$ for $1\leq k\leq j\leq i\leq r$ and $T_0:=1.$
Since $(\sigma_1,\ldots,\sigma_r)\in\R^r\cap U_r$,
there exists $\epsilon>0$ such that $\sigma_1+\cdots+\sigma_i-i>\epsilon$ for all $1\leq i\leq r$. 
Therefore, $\sum_{n_1\geq \cdots\geq n_r\geq N}\frac{1}{n_1^{\sigma_1}\cdots n_r^{\sigma_r}}=O(N^{-\epsilon})$ as $N\rightarrow\infty$. 
This completes the proof.

\subsection{Proof of \propref{trans}}
We prove this by induction on $r$. First we assume that $z_1 \ne 1$. Let $r\geq1$ and $n_1\geq2$ be integers. Then for any $s_1\in \C$, we have
\begin{equation}\label{start}
\frac{1}{n_1^{s_1-1}}-\frac{1}{(n_1+1)^{s_1-1}}=\sum_{k\geq0}(-1)^k\frac{(s_1-1)_{k+1}}{(k+1)!}\frac{1}{n_1^{s_1+k}}.
\end{equation}
Now for $(s_1,\ldots,s_r)\in U_r$, if we multiply both sides of the above equation by
$z_1^{n_1}\frac{z_{2}^{n_{2}}\cdots z_r^{n_r}}{n_{2}^{s_{2}}\cdots n_r^{s_r}}$ and sum for $ n_1\geq \cdots \geq n_r\geq N$, we get
\begin{align}\label{trans-*}
\begin{split}
&\sum_{n_1\geq \cdots \geq n_r\geq N}\left(\frac{z_1^{n_1}z_{2}^{n_{2}}\cdots z_r^{n_r}}{n_1^{s_1-1}n_{2}^{s_{2}}\cdots n_r^{s_r}}-\frac{z_1^{n_1}
z_{2}^{n_{2}}\cdots z_r^{n_r}}{(n_1+1)^{s_1-1}n_{2}^{s_{2}}\cdots n_r^{s_r}}\right)\\
=&\sum_{n_1\geq \cdots \geq n_r\geq N}\sum_{k\geq 0}\frac{(-1)^{k}(s_1-1)_{k+1}}{(k+1)!}\frac{z_1^{n_1}z_{2}^{n_{2}}\cdots z_r^{n_r}}{n_1^{s_1+k}n_{2}^{s_{2}}\cdots n_r^{s_r}}.
\end{split}
\end{align}
 When $r=1$, then the left-hand side of the above equation is,
 \begin{align}\label{r=1}
 \begin{split}
 \sum_{n_1\geq N}\left(\frac{z_1^{n_1}}{n_1^{s_1-1}}-\frac{z_1^{n_1}}{(n_1+1)^{s_1-1}}\right)
 &=\sum_{n_1\geq N}\frac{z_1^{n_1}}{n_1^{s_1-1}}-\frac{1}{z_1}\sum_{n_1\geq N}\frac{z_1^{n_1}}{n_1^{s_1-1}}+\frac{1}{z_1}\frac{z_1^{N}}{N^{s_1-1}}\\
 &=\frac{z_1^{N-1}}{N^{s_1-1}}+\left(1-\frac{1}{z_1}\right)\sum_{n_1\geq N}\frac{z_1^{n_1}}{n_1^{s_1-1}}\\
 &=\frac{z_1^{N-1}}{N^{s_1-1}}+\left(1-\frac{1}{z_1}\right)\mathrm{Li}^{\star}_{(z_1)}(s_1-1)_{\geq N}.
 \end{split}
 \end{align}
 Note that if $z_1=1$, this is nothing but $1/N^{s_1-1}$.
 Since $s_1\in U_1$, using \lemref{uni-convg}, the right-hand side of \eqref{trans-*}, for $r=1$, becomes
 \begin{align*}
 \sum_{k\geq 0}(-1)^k\frac{(s_1-1)_{k+1}}{(k+1)!}\sum_{n_1\geq N}\frac{z_1^{n_1}}{n_1^{s_1+k}}= \sum_{k\geq 0}(-1)^k\frac{(s_1-1)_{k+1}}{(k+1)!}\mathrm{Li}^{\star}_{(z_1)}(s_1+k)_{\geq N}.
 \end{align*}
This completes the proof in the case $r=1$. Now assume that $r\geq 2$. Then as before,
the left-hand side of the equation \eqref{trans-*} is
 \begin{align*}
& \sum_{n_{2}\geq \cdots \geq n_r\geq N}\sum_{n_1\geq n_{2}}\left(\frac{z_1^{n_1}}{n_1^{s_1-1}}-\frac{z_1^{n_1}}{(n_1+1)^{s_1-1}}\right)\frac{z_{2}^{n_{2}}
\cdots z_r^{n_r}}{n_{2}^{s_{2}}\cdots n_r^{s_r}}\\
&=\sum_{n_{2}\geq \cdots \geq n_r\geq N}\left(\frac{z_1^{n_{2}-1}}{n_{2}^{s_1-1}}+
\left(1-\frac{1}{z_1}\right)\sum_{n_1\geq n_{2}}\frac{z_1^{n_1}}{n_1^{s_1-1}}\right)\frac{z_{2}^{n_{2}}
\cdots z_r^{n_r}}{n_{2}^{s_{2}}\cdots n_r^{s_r}}\\
&=\frac{1}{z_1}\sum_{n_{2}\geq \cdots \geq n_2\geq N}\frac{(z_{[1,2]})^{n_{2}}z_{3}^{n_{3}}\cdots z_r^{n_r}}{n_{2}^{s_1+s_{2}-1}n_{3}^{s_{3}}
\cdots n_r^{s_r}}+\left(1-\frac{1}{z_1}\right) \sum_{n_{1}\geq \cdots \geq n_r\geq N}\frac{z_{1}^{n_{1}}z_{2}^{n_{2}}\cdots z_r^{n_r}}{n_{1}^{s_{1}-1}n_{2}^{s_{2}}\cdots n_r^{s_r}}\\
&=\frac{1}{z_1}\mathrm{Li}^{\star}_{(z_{[1,2]},z_{2}\ldots,z_r)}(s_1+s_{2}-1,s_{3},\ldots,s_r)_{\geq N}
+\left(1-\frac{1}{z_1}\right)\mathrm{Li}^{\star}_{\bf z}(s_1-1,s_{2},\ldots,s_r)_{\geq N}.
 \end{align*}
Note that if $z_1=1$, this is nothing but $\mathrm{Li}^{\star}_{(z_{[1,2]},z_{2},\ldots,z_r)}(s_1+s_{2}-1,s_{3},\ldots,s_r)_{\geq N}$.
Again, using \lemref{uni-convg}, the right-hand side of the equation \eqref{trans-*} is 
 $$
 \sum_{k\geq0}(-1)^k\frac{(s_1-1)_{k+1}}{(k+1)!}\mathrm{Li}^{\star}_{\bf z}(s_1+k,s_2,\ldots,s_r)_{\geq N}.
 $$
 This completes the proof of \propref{trans}.

 
\subsection{Proof of \lemref{boundary-1}}
If $q({\mathbf{z}})=r+1$, i.e., $z_1=\cdots=z_r=1$, then $\mathrm{Li}^{\star}_{(z_r,\ldots,z_1)}(s_r,\ldots,s_1)_{\geq N}=\zeta^{\star}(s_r,\ldots,s_1)_{\geq N}$. 
Therefore, this case is done by \cite[Lemma 3]{BS2}. So we assume that there exists an integer $1\leq i\leq r$ such that $z_{i}\neq 1$.
Therefore, we have $a_1+\cdots+a_r>r-1$ from the hypothesis.
Now to prove the required assertion, we will use the double induction, first on the depth $r$ and then on the smallest
non-negative integer $k_0$ such that $(a_r+k_0,a_{r-1},\ldots,a_1)\in U_r$.
Note that if $k_0=0$, the result is immediate by \lemref{uni-convg}, hence we suppose $k_0 \ge 1$.

Assume that $r=1$. This implies that $z_1\neq 1$, $a_1>0$. Now by using the translation formula \eqref{eq1} 
and \rmkref{series-tail}, we have a disc $D$ around $a_1$ such that the sequence of meromorphic functions $(v_N)_{N\geq 2}$, given by
$$
v_N:= \left(1-\frac{1}{z_1}\right)\mathrm{Li}^{\star}_{(z_1)}(s_1)_{\geq N} + \frac{z_1^{N-1}}{N^{s_1}},
$$
is holomorphic on $D$ and converges normally to $0$ on $D$, as $N\rightarrow\infty$. Since $\frac{z_1^{N-1}}{N^{s_1}}$ 
converges normally to $0$ on $D$ (as $N\rightarrow\infty$), for the polynomial $P(s_1):=1$,
we get the required assertion for $r=1$.

Now assume that $r\geq 2$. Here using the translation formula \eqref{eq2} and \rmkref{series-tail}, 
we have a polydisc $D_1$ around ${\bf a}$ and the sequence of meromorphic function $(v_N)_{N\geq 2}$, given by
 \begin{align}\label{medstep-1}
 \begin{split}
v_N:=&\left(1-\frac{1}{z_r}\right)\mathrm{Li}^{\star}_{(z_r,z_{r-1},\ldots,z_1)}(s_r-1,s_{r-1},\ldots,s_1)_{\geq N}\\&+
\frac{1}{z_r}\mathrm{Li}^{\star}_{(z_{[r-1,r]},z_{r-2},\ldots,z_1)}(s_r+s_{r-1}-1,s_{r-2},\ldots,s_1)_{\geq N}\\
&-\sum_{k=0}^{k_0-1}(-1)^k\frac{(s_r-1)_{k+1}}{(k+1)!}\mathrm{Li}^{\star}_{(z_r,\ldots,z_1)}(s_r+k,s_{r-1},\ldots,s_1)_{\geq N},
\end{split}
\end{align}
is holomorphic on $D_1$ and converges normally to $0$ on $D_1$, as $N\rightarrow\infty$. Here $k_0$ is as above,
the smallest non-negative integer such that $(a_r+k_0,a_{r-1},\ldots,a_1)\in U_r$.

Now, first consider the case when $z_r=1$. Note that in this case it is enough to prove the result for 
$\mathrm{Li}^{\star}_{(z_{[r-1,r]},z_{r-2},\ldots,z_1)}(s_r+s_{r-1}-1,s_{r-2},\ldots,s_1)_{\geq N}$ and 
$\mathrm{Li}^{\star}_{(z_r, z_{r-1},\ldots,z_1)}(s_r+k, s_{r-1},\ldots,s_1)_{\geq N}$ with 
$1 \le k \le k_0-1$, as the first term in the expression of $v_N$ is $0$ in this case.
Note that $z_{[r-1,r]}=z_{r-1}$ and  $a({\bf z})_r=a({\bf z})_{r-1}$.  Therefore,
$$
a_1+\cdots+a_{r-2}+a_{r-1}+a_{r}-1>(r-1)-a(\mathbf{z})_{r-1}.
$$
Moreover, $a_1+\cdots+a_r+k>r-a(\mathbf{z})_r$ for each $1 \le k \le k_0-1$.
Hence using the induction hypothesis for depth $r-1$ and also for depth $r$ (with $k <k_0$), we get a 
polydisc $D_2$ around ${\bf a}$ and a polynomial $P_1(s_1,\ldots,s_r)$ such that
the sequences of holomorphic functions on $D_2$,
$$
(P_1(s_1,\ldots,s_r)\mathrm{Li}^{\star}_{(z_{[r-1,r]},z_{r-2},\ldots,z_1)}(s_r+s_{r-1}-1,s_{r-2},\ldots,s_1)_{\geq N})_{N \ge 2}
$$
and 
$$
(P_1(s_1,\ldots,s_r)\mathrm{Li}^{\star}_{(z_r, z_{r-1},\ldots,z_1)}(s_r+k,s_{r-1}, \ldots,s_1)_{\geq N})_{N \ge 2}
$$
converge normally to $0$ on $D_2$, as $N \to \infty$, for each $1 \le k \le k_0-1$.
Now take $P(s_1,\ldots,s_r)=(s_r-1) P_1(s_1,\ldots,s_r)$. 
This is the desired polynomial to conclude the lemma in this case, by taking $D=D_1 \cap D_2$.

Next let $z_r\neq 1$. In that case, we consider the sequence of meromorphic functions $(v_N)_{N\geq 2}$
as in \eqref{medstep-1}, by replacing the variable $s_r$ by $s_r+1$
and the sum over $k$ runs for $0 \le k \le k_0-2$. The proof works out identically as before, and hence we omit the details.
This completes the proof of \lemref{boundary-1}.

\subsection{Proof of \lemref{boundary-2}}
Note that if $t_r=0$, i.e., $z_1=\cdots=z_r=1$, the lemma follows from \cite[Lemma 3]{BS2}.
Hence we assume that $t_r \ge 1$.
Now assume that $r=1$. Then $z_1\neq 1$, $a_1=0$ and $k_0=2$. So using the translation formula 
\eqref{eq1} and \rmkref{series-tail}, we have a disc $D_1$ around $a_1$ such that the sequence of 
functions $(v_N)_{N\geq 2}$, given by
$$
v_N:=\frac{z_1^{N-1}}{N^{s_1}}+\left(1-\frac{1}{z_1}\right)\mathrm{Li}^{\star}_{(z_1)}(s_1)_{\geq N}-s_1\mathrm{Li}^{\star}_{(z_1)}(s_1+1)_{\geq N},
$$
is holomorphic on $D_1$ and converges normally to $0$ on $D_1$, as $N \to\infty$. Since $a_1+1>0$, 
applying \lemref{boundary-1} to the function $\mathrm{Li}^{\star}_{(z_1)}(s_1+1)_{\geq N}$, 
we get a disc $D_2$ around $a_1$ and a polynomial $P(s_1)$ such that $P(s_1)\mathrm{Li}^{\star}_{(z_1)}(s_1)_{\geq N}$ is 
holomorphic on $D_2$ and $\left\|P(s_1)\mathrm{Li}^{\star}_{(z_1)}(s_1+1)_{\geq N}\right\|_{D_2}\rightarrow 0$, as $N \to \infty$.
Hence
$$
\left\|P(s_1)\mathrm{Li}^{\star}_{(z_1)}(s_1)_{\geq N}-\frac{P(s_1)z_1^N N^{-s_1}}{(1-z_1)}\right\|_D\rightarrow 0,
$$
as $N \to \infty$, where $D=D_1 \cap D_2$.

Next let $r\geq 2$. We first assume that $z_r=1$. Now similar to the proof of \lemref{boundary-1},
there exists a polydisc $D_1$ around ${\bf a}$ and a non-negative integer $k_0$, least with $(a_r+k_0,a_{r-1},\ldots,a_1)\in U_r$, 
such that the sequence of holomorphic function $(v_N)_{N\geq 2}$ on $D_1$, defined in \eqref{medstep-1},
converges normally to 0 on $D_1$. Since for each $1 \le k \le k_0-1$, we have $a_1+\cdots+a_r+k=r-1+k>r-1$, 
therefore applying \lemref{boundary-1}, we get a polydisc $D_2$ around ${\bf a}$ and a polynomial $P_1(s_1,\ldots,s_r)$
such that the sequence of holomorphic functions on $D_2$,
$$
(P_1(s_1,\ldots,s_r)\mathrm{Li}^{\star}_{(z_r, z_{r-1},\ldots,z_1)}(s_r+k, s_{r-1}, \ldots,s_1)_{\geq N})_{N \ge 2}
$$
converges normally to $0$ on $D_2$, as $N \to\infty$.
Here $a({\bf z})_r=a(z_1,\ldots,z_{r-2},z_{[r-1,r]})_{r-1}=1$, as $z_{[r-1,r]}=z_{r-1}$. 
Therefore, using the induction hypothesis for the function (of depth $r-1$) $\mathrm{Li}^{\star}_{(z_{[r-1,r]},z_{r-2},\ldots,z_1)}(s_r+s_{r-1}-1,s_{r-2},\ldots,s_1)_{\geq N}$, we get a polydisc $D_3$ around ${\bf a}$ and
a polynomial $P_2(s_1,\ldots,s_r)$ such that the sequence of functions $(w_N)_{N \ge 2}$, given by
 \begin{align*}
w_N:=& P_2(s_1,\ldots,s_r)\mathrm{Li}^{\star}_{(z_{[r-1,r]},z_{r-2},\ldots,z_1)}(s_r+s_{r-1}-1,s_{r-2},\ldots,s_1)_{\geq N}\\
&-\frac{P_2(s_1,\ldots,s_r)z_{[1,r]}^NN^{-(s_1+\cdots+s_r+1-r)}}{\prod_{1\leq j\leq r-1} H_{j}^{({r-1})}(s_j, \ldots,s_{r-2},s_r+s_{r-1}-1;z_{[t_{r-1},r-1]})},
\end{align*}
is holomorphic on $D_3$ and converges normally to $0$ on $D_3$, as $N \to\infty$.
If we take $P_3(s_1,\ldots,s_r) = P_1(s_1,\ldots,s_r) P_2(s_1,\ldots,s_r)$,
then we get that $P_3(s_1,\ldots,s_r)(s_r-1)\mathrm{Li}^{\star}_{(z_r,\ldots,z_1)}(s_r,\ldots,s_1)_{\geq N}$ is 
a holomorphic function on $D=D_1 \cap D_2 \cap D_3$ and as $N \to \infty$, the sequence of functions
$(W_N)_{N \ge 2}$, given by 
\begin{align*}
W_N:=&P_3(s_1,\ldots,s_r)(s_r-1)\mathrm{Li}^{\star}_{(z_r,\ldots,z_1)}(s_r,\ldots,s_1)_{\geq N}\\
&-\frac{P_3(s_1,\ldots,s_r) (s_r-1) z_{[1,r]}^NN^{-(s_1+\cdots+s_r+1-r)}}{(s_r-1)\prod_{1\leq j\leq r-1} H_{j}^{({r-1})}(s_j, \ldots,s_{r-2},s_r+s_{r-1}-1;z_{[t_{r-1},r-1]})},
\end{align*}
is holomorphic on $D$ and converges normally to $0$ on $D$, as $N \to \infty$.
Since $z_{[r-1,r]}=z_{r-1}$, we have $t_{r-1}=t_{r}$ and $z_{[t_{r-1},r-1]}=z_{[t_{r},r]}$. 
Hence it also follows that
$$
H_{j}^{({r-1})}(s_j, \ldots,s_{r-2},s_r+s_{r-1}-1;z_{[t_{r-1},r-1]})=H_{j}^{({r})}(s_j,\ldots,s_r;z_{[t_{r},r]}),
$$
for all $1\leq j\leq r-1$.
Now, by definition, $H_{r}^{(r)}(s_r;z_{[t_r,r]})=s_r-1$, since $t_r\neq r$ as $z_r=1$.
Hence taking $P(s_1,\ldots,s_r)=(s_r-1)P_3(s_1,\ldots,s_r)$, we get the required result in this case.

Now assume that $z_r\neq 1$. Again, as in the proof of \lemref{boundary-1}, 
we consider the sequence of meromorphic function $(v_N)_{N\geq 2}$, as in \eqref{medstep-1}, but
replacing the variable $s_r$ by $s_r+1$ and the sum over $k$ runs for $0 \le k \le k_0-2$.
Thus the steps in the previous case till obtaining the polydisc $D_2$ and  polynomial $P_1$,
work analogously. Now note that in this case, since $z_r\neq 1$, we have $t_r=r$ and hence
$z_1=\cdots=z_{r-2}=1$ and $z_{r-1}z_r=1$. Note that here
$\mathrm{Li}^{\star}_{(z_{[r-1,r]},z_{r-2},\ldots,z_1)}(s_r+s_{r-1},s_{r-2},\ldots,s_1)_{\geq N}
=\zeta^{\star}(s_r+s_{r-1},s_{r-2},\ldots,s_1)$. Now by the induction hypothesis
(or \cite[Lemma 3]{BS2}) we get a polydisc $D_3$ around ${\bf a}$ and
a polynomial $P_2(s_1,\ldots,s_r)$ such that the sequence of functions $(w_N)_{N \ge 2}$, given by
\begin{align*}
w_N:=P_2(s_1,\ldots,s_r)\zeta^{\star}(s_r+s_{r-1},s_{r-2},\ldots,s_1)-
\frac{P_2(s_1,\ldots,s_r) N^{-(s_1+\cdots+s_r+1-r)}}{\prod_{1\leq j\leq r-1}(s_j+\cdots+s_r-(r-j))},
\end{align*}
is holomorphic on $D_3$ and converges normally to $0$ on $D_3$, as $N \to\infty$.
So for $D=D_1 \cap D_2 \cap D_3$ and
$P(s_1,\ldots,s_r)=P_1(s_1,\ldots,s_r)P_2(s_1,\ldots,s_r)$, in view of the above analysis and
\eqref{medstep-1} (with $s_r$ replaced by $s_r+1$), we get that
$P(s_1,\ldots,s_r)\mathrm{Li}^{\star}_{(z_r,\ldots,z_1)}(s_r,\ldots,s_1)_{\geq N}$ 
is holomorphic on $D$ and
\begin{align*}
\left\|P(s_1,\ldots,s_r)\mathrm{Li}^{\star}_{(z_r,\ldots,z_1)}(s_r,\ldots,s_1)_{\geq N}
-\frac{P(s_1,\ldots,s_r)z_{[1,r]}^NN^{-(s_1+\cdots+s_r+1-r)}}
{(1-z_r)\prod_{1\leq j\leq r-1}(s_j+\cdots+s_r-(r-j))}\right\|_D\rightarrow 0,
\end{align*}
as $N \to \infty$, since $z_{[1,r]}=1$ here. This completes the proof of \lemref{boundary-2}.

\subsection{Proof of \lemref{boundary-3}}
Let $i_1,i_2$ be two integers such that $1\leq i_1<i_2\leq r$ and $z_{[i_1,r]},z_{[i_2,r]}\neq 1$. 
By definition $a(\mathbf{z})_r=1$, hence by the hypothesis of lemma, $a_1+\cdots+a_r=r-1$. 
We prove this lemma by induction on $r$.

 Assume that $r=2$. Then $z_{[1,2]}\neq 1$ and $z_2\neq 1$. We use the following version of \eqref{eq2}:
 $$
 \left(1-\frac{1}{z_2}\right)\mathrm{Li}^{\star}_{(z_2,z_1)}(s_2,s_{1})_{\geq N}+
\frac{1}{z_2}\mathrm{Li}^{\star}_{(z_{[1,2]})}(s_1+s_{2})_{\geq N}
=\sum_{k\geq0}(-1)^k\frac{(s_2)_{k+1}}{(k+1)!}\mathrm{Li}^{\star}_{(z_2,z_1)}(s_2+k+1,s_1)_{\geq N}.
 $$
 Applying steps similar to those in the proofs of \lemref{boundary-1} and 
 \lemref{boundary-2}, we get a polydisc $D_1$ and a polynomial $P_1(s_1,s_2)$
around $(a_1,a_2)$ such that the sequence of functions $(v_N)_{N \ge 2}$, given by
$$
v_N:= P_1(s_1,s_2)\left(1-\frac{1}{z_2}\right)\mathrm{Li}^{\star}_{(z_2,z_1)}(s_2,s_{1})_{\geq N}+
\frac{P_1(s_1,s_2)}{z_2}\mathrm{Li}^{\star}_{(z_{[1,2]})}(s_1+s_{2})_{\geq N},
$$
is holomorphic in $D_1$ and converges normally to $0$ on $D_1$, as $N \to \infty$.
Note that  as  $z_{[1,2]}\neq 1$ and $a_1+a_2=1>0$, by applying \lemref{boundary-1} to the function
$\mathrm{Li}^{\star}_{(z_{[1,2]})}(s_2+s_1)_{\geq N}$, we get a polydisc $D_2$ and a polynomial $P_2(s_1,s_2)$
such that $P_2(s_1,s_2)\mathrm{Li}^{\star}_{(z_{[1,2]})}(s_2+s_1)_{\geq N}$ is holomorphic on $D_2$ and as $N \to \infty$,
 $$
 \left\|P_2(s_1,s_2)\mathrm{Li}^{\star}_{(z_{[1,2]})}(s_2+s_1)_{\geq N}\right\|_{D_2}\rightarrow 0.
 $$
 This completes the proof for $r=2$, by taking $D=D_1 \cap D_2$ and $P(s_1,s_2)=P_1(s_1,s_2)P_2(s_1,s_2)$.
 
Now assume that $r\geq 3$. Again, we break the proof into two cases: $z_r=1$ and $z_r\neq 1$. 
The proofs of Lemma \ref{boundary-1}, \ref{boundary-2} and the $r=2$ case discussed above indicate that the proof
reduces to finding a polydisc $D$ around ${\bf a}$ and a polynomial $P(s_1,\ldots,s_r)$ such that as $N \to \infty$,
$$
 \left\|P(s_1,\ldots,s_r)\mathrm{Li}^{\star}_{(z_{[r-1,r]},z_{r-2},\ldots,z_1)}(s_r+s_{r-1}-1,s_{r-2},\ldots,s_1)_{\geq N},
 \right\|_D\rightarrow 0
$$
or,
$$
 \left\|P(s_1,\ldots,s_r)\mathrm{Li}^{\star}_{(z_{[r-1,r]},z_{r-2},\ldots,z_1)}(s_r+s_{r-1},s_{r-2},\ldots,s_1)_{\geq N}
 \right\|_D\rightarrow 0,
$$
depending on whether $z_r=1$ or not.
 
First suppose $z_r=1$. Note that $a_1+\cdots+a_r-1=(r-1)-1$.
Since $z_r=1$, we have at least two integers $i_1, i_2$ such that $z_{[i_1,r-1]},z_{[i_2,r-1]}\neq 1$. 
Hence this allows us to use the induction hypothesis on depth $r-1$, to get the desired conclusion in this case.
 
 Next assume that $z_r\neq 1$. By the hypothesis of the lemma, 
 there exists an integer $i$ such that $1\leq i \leq r-1$ and $z_{[i,r]}\neq 1$. Therefore,
 we have $a(z_1,\ldots,z_{r-2},z_{[r-1,r]})_{r-1}=1$. 
 Hence as $a_1+\cdots+a_r=r-1>(r-1)-1$, we can apply \lemref{boundary-1} to get the desired conclusion in this case.
 This completes the proof of \lemref{boundary-3}.

\subsection{Proof of \lemref{Taylor-lem}}
For every integer $N\geq 1$, let the Taylor series expansion of
$u_N(s_1,\ldots,s_r)$ around ${\bf a}$ be written as
$$
\sum_{k_1,\ldots,k_r\geq 0}\alpha_{(k_1,\ldots,k_r)}(N)(s_1-a_1)^{k_1}\cdots(s_r-a_r)^{k_r}.
$$ 
Then for every ${\bf k}=(k_1,\ldots,k_r) \in \N^r$, $\alpha_{\bf k}(N)=P_{\bf k}(N,\log N)+o(1)$, as $N\rightarrow\infty$. 
Let the Taylor series expansion of the function $u(s_1,\ldots,s_r)$ around ${\bf a}$ be written as
$$
\sum_{k_1,\ldots,k_r\geq 0}\beta_{(k_1,\ldots,k_r)}(s_1-a_1)^{k_1}\cdots(s_r-a_r)^{k_r}.
$$
Since $u_N(s_1,\ldots,s_r)$ converges uniformly to the holomorphic function $u(s_1,\ldots,s_r)$ on $D$,
we get that $P_{\bf k}(N,\log N)+o(1)$ converges to $\beta_{\bf k}$ as $N\rightarrow\infty$. 
Now given that $P_{\bf k}(N,\log N)=\sum_{i,j \ge 0}a_{\bf k}^{(i,j)}(N)N^{i}(\log N)^j$, 
the sequence $P_{\bf k}(N,\log N)$ is convergent if and only if we have $a_{\bf k}^{(i,j)}(N)=0$ for all 
pair $(i,j)\neq (0,0)$ and $(a_{\bf k}^{(0,0)}(N))_{N\geq 1}$ is an eventually constant sequence of
$\mathcal{C}_{\bf z}$. Assume that for $N$ large enough, $a_{\bf k}^{(0,0)}(N)=a_{\bf k}$.
Then $a_{\bf k}$ is same as  $\beta_{\bf k}$, which is exactly the ${\bf k}$-th
coefficient in the Taylor series expansion around ${\bf a}$
of the function $u(s_1,\ldots,s_r)$. This completes the proof of \lemref{Taylor-lem}.

\section{Asymptotic expansions of sequences of germs of holomorphic
and meromorphic functions with variable coefficients}\label{sec-asymp}

Let $\mathcal{E}$ be the comparison scale,
on the set $\N$, filtered by the Fr\'echet filter and formed by the sequences 
$$
\left( (\log n)^ln^{-m} \right)_{n \geq 1},
$$
where $l \in \N$ and $m \in \Z$. Also, let $r \ge 1$ be an integer and
let ${\bf z}=(z_1,\ldots,z_r) \in \C^r$ be such that $z_1,\ldots,z_r$ are fixed roots of unity.
As before, let $\mathcal{C}_{\bf z}$ be the $\C$-algebra
generated by the constant sequence $(1)_{n\geq 1}$ and the sequences of the form $(z_i^n)_{n\geq 1}$ for $1 \le i \le r$.
Note that as a $\C$-vector space, ${\mathcal C}_{\bf z}$ is finite dimensional since the set
$$
\mathcal S:=\{z_1^{k_1} \cdots z_r^{k_r}: k_i \in \N \text{ for } 1 \le i \le r\}
$$
is finite (see \cite[\S 2]{PSBS}). For every element $\xi \in \mathcal S$, there are infinitely many $(k_1, \ldots, k_r) \in \N^r$
such that $\xi=z_1^{k_1} \cdots z_r^{k_r}$. By the exponent of $\xi$, we mean the smallest such
element in $\N^r$, as per the dictionary ordering. Let $S$ denote the set of exponents of elements of $\mathcal S$.
Then $\mathcal C_{\bf z}$ has a basis of the form
\begin{equation}\label{basis}
\mathcal B_{\bf z}=\{ (z_1^{k_1 n} \cdots z_r^{k_r n})_{n \ge 1} : (k_1, \ldots, k_r) \in S\}.
\end{equation}
As $(1)_{n \ge 1} \in \mathcal B_{\bf z}$, we have $(0, \ldots,0) \in S$. It is immediate that $\mathcal B_{\bf z}$
is a spanning set for $\mathcal C_{\bf z}$. Moreover,
as we are only taking $(k_1, \ldots, k_r) \in S$, we get the linear independence
using a Vandermonde determinant trick.

\subsection{Asymptotic expansions of sequences of germs of holomorphic functions with variable coefficients}\label{holo-asym}

Majority of the notations used here are taken from \cite[\S 7.1]{BS2}, where the notion of asymptotic expansions of sequences
of germs of holomorphic functions is explained. We extend that to the notion of asymptotic expansions of sequences
of germs of holomorphic functions with variable coefficients as follows:

For fixed $\mathbf{a}=(a_1,\ldots,a_r) \in \C^r$, let $\mathcal{O}_{\mathbf{a}}$ denote the
$\C$-algebra of germs of holomorphic functions at  $\mathbf{a}$. For simplicity, we use the same
notation for the function and its germ at ${\bf a}$.
For $A\in \Z$, we say that a sequence $(f_n)_{n\geq1}$ of germs of holomorphic functions at $\mathbf{a}$ has an 
{\it asymptotic expansion to precision $n^{-A}$} relative to $\mathcal{E}$ with coefficients in $\mathcal{C}_{\bf z}$,  
if the following conditions holds:
\begin{enumerate}
\item[(a)] Suppose that in a neighbourhood of ${\bf a}$, the Taylor series expansion of $f_n$ is written as
$$
f_n(s_1,\ldots,s_r)=\sum_{k_1,\ldots,k_r \ge 0} c_{(k_1,\ldots,k_r)}(f_n) (s_1-a_1)^{k_1} \cdots (s_r-a_r)^{k_r}.
$$
Then for each $\mathbf{k}=(k_1,\ldots,k_r)$, there exists a family of sequence
$((u_{(\mathbf{k},l,m)}(n))_{n \ge 1})_{l\in \N, m\in \Z}$ in $\mathcal{C}_{\bf z}$ 
such that the sequence $(c_{\mathbf{k}}(f_n))_{n\geq 1}$ has the asymptotic expansion to precision $n^{-A}$ relative to
$\mathcal{E}$ with coefficients in $\mathcal{C}_{\bf z}$ (in the sense of \cite [\S 2]{PSBS}) of the form
$$
c_{\mathbf{k}}(f_n)=\sum_{\substack{l\in \N, m\in \Z\\ m\leq A}}u_{(\mathbf{k},l,m)}(n)(\log n)^{l}n^{-m}+O(n^{-A}),
$$
as $n \to \infty$, where the sequence $(u_{(\mathbf{k},l,m)}(n))_{n\geq 1}$ is the zero sequence for all but finitely many pairs $(l,m)$;
\item[(b)] there exists $m_0\in \Z$ such that $u_{(\mathbf{k},l,m)}(n)=0$ for all $n\geq 1$, $\mathbf{k}\in \N^r$, $l\in\N $ and $m\leq m_0$;
\item[(c)] for each $(l, m)\in \N\times\Z$ and $\mathbf{k}\in \N^r$, if $u_{(\mathbf{k},l,m)}$ denotes the coordinate corresponding to the
 constant sequence $(1)_{n\geq 1}$ with respect to the basis $\mathcal{B}_{\bf z}$ for the sequence
 $(u_{(\mathbf{k},l,m)}(n))_{n\geq 1}$,
 then the power series
\begin{align}
g_{(l,m)}:=\sum_{k_1,\ldots,k_r \ge 0}u_{(\mathbf{k},l,m)}(s_1-a_1)^{k_1} \cdots (s_r-a_r)^{k_r},
\end{align}
converges in a neighbourhood of $\mathbf{a}$.
\end{enumerate}

When these conditions hold, then the formal series
$$
\sum_{\substack{l\in \N, m\in \Z \\ m\leq A }}g_{(l,m)}L^lX^m,
$$
is called the {\it formal asymptotic expansion to precision $n^{-A}$} (at ${\bf a}$) relative to $\mathcal{E}$ with coefficients in 
$\mathcal{C}_{\bf z}$ of the sequence $(f_n)_{n\geq 1}$.

\begin{ex}\label{zero-asymp}\rm
Let $(f_n)_{n\geq 1}\subseteq \mathcal{O}_{\mathbf{a}}$ be such that there exists an 
open neighbourhood $D$ of $\mathbf{a}$ such that $\| f_n\|_D=o(n^{-A})$ 
for some integer $A$, as $n \to \infty$. Then $(f_n)_{n\geq 1}$
has an asymptotic expansion to precision $n^{-A}$ relative to $\mathcal{E}$ with coefficients in 
$\mathcal{C}_{\bf z}$ and its formal asymptotic expansion to precision $n^{-A}$ 
relative to $\mathcal{E}$ with coefficients in 
$\mathcal{C}_{\bf z}$ is the Laurent polynomial $0$.
\end{ex}

We say that a sequence $(f_n)_{n\geq 1}$ of elements of $\mathcal{O}_{\mathbf{a}}$ has a 
{\it complete asymptotic expansion} (at ${\bf a}$) relative to $\mathcal{E}$ with 
coefficients in $\mathcal{C}_{\bf z}$ if it has an asymptotic expansion to precision $n^{-A}$
relative to $\mathcal{E}$ with coefficients in 
$\mathcal{C}_{\bf z}$, for all $A\in \Z$. 
In this case, there exists a unique Laurent series
$$
G=\sum_{(l,m)\in \N\times\Z}g_{(l,m)}L^lX^m,
$$ 
in the indeterminate $X$ with coefficients in the formal power series ring $\mathcal{O}_{\mathbf{a}}[[L]]$ such that the 
Laurent polynomial obtained by truncating $G$ to degree $\leq A$ in $X$ is the formal asymptotic expansion of 
$(f_n)_{n\geq 1}$ to precision $n^{-A}$ relative to $\mathcal{E}$ with coefficients in $\mathcal{C}_{\bf z}$.
We call $G$ to be the {\it formal complete asymptotic expansion} (at ${\bf a}$) of
$(f_n)_{n\geq 1}$ (relative to $\mathcal{E}$ with coefficients in $\mathcal{C}_{\bf z}$).

\begin{ex}\label{g-r=1}\rm
Let $a\in \Z$. For each integer $n\geq1$, consider the germs of holomorphic functions $f_n$ defined by 
$s \mapsto z^n n^ {1-s}$, where $z$ is a root of unity. So around $a$, $f_n$ is defined by the power series,
$$
f_n(s)=\sum_{k\geq0}\frac{(-1)^k(\log n)^k z^n n^{1-a}}{k!}(s-a)^k.
$$
When $z=1$, we get, $g_{(l,a-1)}=\frac{(-1)^l}{l!}(s-a)^{l}$ and $g_{(l,m)}=0$ for all $m\neq a-1$. But
when $z\neq 1$, we get $g_{(l,m)}=0$ for all $l\in \N, m\in \Z$. Therefore, in this case, the formal 
complete asymptotic expansion of $(f_n)_{n\geq 1}$ is zero.
\end{ex}

To conclude this segment we need to also define the notion of asymptotic expansions of sequences
of germs of meromorphic functions with variable coefficients. For that we need the following lemma,
which is an extension of \cite[lemma 7]{BS2}. Since the proof follows analogously, we omit the proof.

\begin{lem}\label{lem-den}
Let $(f_n)_{n\geq1}$ be a sequence in $\mathcal{O}_{\mathbf{a}}$ and $f$ be a 
non-zero element in $\mathcal{O}_{\mathbf{a}}$. If the sequence $(ff_n)_{n\geq1}$ has an 
asymptotic expansion to precision $n^{-A}$, for a given integer $A$ (resp. a complete asymptotic expansion) 
relative to $\mathcal{E}$ with coefficients in $\mathcal{C}_{\bf z}$, then the same 
holds for the sequence $(f_n)_{n\geq 1}$.
\end{lem}

\subsubsection{Asymptotic expansions of sequences of germs of meromorphic functions with variable coefficients}\label{asymp-mero}

Let $\mathcal{M}_{\mathbf{a}}$ denote the $\C$-algebra of germs of meromorphic 
functions at $\mathbf{a}$.  For an integer $A\in \Z$, we say that a sequence $(f_n)_{n\geq 1}$ 
of elements of $\mathcal{M}_{\mathbf{a}}$ has an {\it asymptotic expansion to precision $n^{-A}$}
(resp. a {\it complete asymptotic expansion})  
relative to  $\mathcal{E}$ with coefficients in $\mathcal{C}_{\bf z}$ if there exists a common denominator $f$ of $f_n$ 
(i.e., a non zero element $f\in \mathcal{O}_{\mathbf{a}}$ such that $ff_n\in \mathcal{O}_{\mathbf{a}}$)
such that the sequence $(ff_n)_{n\geq 1}$ 
has an {asymptotic expansion to precision $n^{-A}$} (resp. a {complete asymptotic expansion})  relative to 
$\mathcal{E}$ with coefficients in $\mathcal{C}_{\bf z}$. 
This definition is independent of the common denominator $f$ by \lemref{lem-den}.
In this case, the formal complete asymptotic expansion of the sequence $(f_n)_{n\geq 1}$  is given by 
\begin{equation}\label{fcae-1}
\sum_{l\in \N, m\in \Z }f^{-1}g_{(l,m)}L^lX^m,
\end{equation}
in the indeterminate $X$ with coefficients in $\mathcal{M}_{\mathbf{a}}[[L]]$, where
$$
\sum_{l\in \N, m\in \Z }g_{(l,m)}L^lX^m,
$$
is the complete asymptotic expansion of $(ff_n)_{n\geq 1}$. This element of the ring $\mathcal{M}_{\mathbf{a}}[[L]]((X))$ 
is then called the {\it formal complete asymptotic expansion} of the sequence of germs of the meromorphic functions 
$(f_n)_{n\geq 1}$ (at ${\bf a}$) relative to $\mathcal{E}$  with  coefficients in $\mathcal{C}_{\bf z}$. Clearly,
restricting $m \le A$ in \eqref{fcae-1}, we get the formal asymptotic expansion to precision $n^{-A}$
of the sequence $(f_n)_{n\geq 1}$.

\begin{rmk}\label{change-variable}{}\rm
Similar to \cite[Remark 16]{BS2}, let $\mathbf{a}\in \C^r$ and $\pi$ be the germ at $\mathbf{a}$ of a 
holomorphic map defined on a neighbourhood of $\mathbf{a}$ with values in $\C^p$ for some $p\geq 0$. 
Denote $\mathbf{b}=\pi(\mathbf{a})$. Let $(f_n)_{n\geq 1}$ be a sequence of elements of 
$\mathcal{M}_{\C^p,\mathbf{b}}$ which has an asymptotic expansion to precision 
$n^{-A}$ (relative to $\mathcal{E}$ with coefficients in $\mathcal{C}_{\bf z}$), with $A\in \Z$. Then 
the sequence $(f_n \circ \pi)_{n\geq 1}$ of elements of $\mathcal{M}_{\C^r,\mathbf{a}}$ also has an 
asymptotic expansion to precision $n^{-A}$ (relative to $\mathcal{E}$ with coefficients in $\mathcal{C}_{\bf z}$).
 Moreover, if $\sum_{\substack{(l,m)\in \N\times\Z\\ m \le A}}g_{(l,m)}L^lX^m$ is the formal asymptotic expansion of 
 $(f_n)_{n\geq1}$ to precision $n^{-A}$, then $\sum_{\substack{(l,m)\in \N\times\Z\\ m\leq A}}(g_{(l,m)}\circ \pi)L^lX^m$ 
 is that of $(f_n\circ \pi)_{n\geq1}$.
\end{rmk}

\subsection{Asymptotic expansion of  $\mathrm{Li}^{\star}_{(z)}(s)_{\geq N}$}
Let $z$ be a root of unity.
The formal asymptotic expansion of the sequence (of germs) of functions $(\mathrm{Li}^{\star}_{(z)}(s)_{\geq N})_{N \ge 2}$
relative to $\mathcal{E}$ with variable coefficients in $\mathcal{C}_{(z)}$, depends on $z$.
If $z=1$, then  $\mathrm{Li}^{\star}_{(z)}(s)_{\geq N}=\zeta^{\star}(s)_{\geq N}$. 
Hence, in this case the formal complete asymptotic expansion (at an integer $a$) is given in \cite[Remark 17]{BS2},
namely, it is the formal Laurent series
$$
\sum_{k \ge 0} \sum_{l \ge 0} h^\star_{(l,k)}L^l X^{a+k-1},
$$
where $h^\star_{(l,k)}(s)=\frac{(-1)^l B_k^\star}{l! k!}(s)_{k-1}(s-a)^l$ and
$B_k^\star$ denotes the star Bernoulli numbers, defined by
$$
\frac{x}{e^x-1}=\sum_{k \ge 0} \frac{(-1)^k B_k^\star}{k!}  x^k.
$$
So, we are now interested in the case when $z\neq 1$ and we prove the following proposition.

\begin{prop}\label{*-mp-asym-r=1}
Let $a$ be an integer and $z\neq 1$ be a root of unity. Then the sequence (of germs at $a$) of the meromorphic functions 
$(\mathrm{Li}^{\star}_{(z)}(s)_{\geq N})_{N\geq 2}$ has a complete asymptotic expansion relative to   
$\mathcal{E}$ with coefficients in $\mathcal{C}_{(z)}$. Moreover, the associated formal complete asymptotic 
expansion is equal to zero.
\end{prop}
\begin{proof}
Let $A$ be a positive integer. We prove this by induction on the smallest non-negative integer $k_0$ such that $a+k_0>A+1$.
Let $D$ be a disc around $a$ with small radius $\rho <1$. So if $k_0=0$, we have $a \ge A+2$ and
hence $\|\mathrm{Li}^{\star}_{(z)}(s)_{\geq N}\|_D=o(N^{-A})$, as $N \to \infty$ and we are done by Example \ref{zero-asymp}.
Now from the translation formula \eqref{eq1}, we can write
$$
\frac{z^{N-1}}{N^s}=\left({\bar z}-1\right)\mathrm{Li}^{\star}_{(z)}(s)_{\geq N}+
\sum_{k\geq 0}(-1)^k\frac{(s)_{k+1}}{(k+1)!}\mathrm{Li}^{\star}_{(z)}(s+k+1)_{\geq N}.
$$
For $k_0 \ge 1$, following the proof of \lemref{uni-convg}, $D$ can be so chosen that as $N \to \infty$, we have
$$
\left\|\sum_{k\geq k_0-1}(-1)^k\frac{(s)_{k+1}}{(k+1)!}\mathrm{Li}^{\star}_{(z)}(s+k+1)_{\geq N}\right\|_D=o(N^{-A}).
$$
Hence, by \exref{zero-asymp}, we get that the sequence (of germs) of holomorphic functions at $a$
$$
\left(\left({\bar z}-1\right)\mathrm{Li}^{\star}_{(z)}(s)_{\geq N}+\sum_{k= 0}^{k_0-2}(-1)^k\frac{(s)_{k+1}}{(k+1)!}\mathrm{Li}^{\star}_{(z)}(s+k+1)_{\geq N}\right)_{N\geq 2}
$$
has an asymptotic expansion to precision $N^{-A}$, with the same formal asymptotic expansion 
to precision $N^{-A}$ as that of the sequence (of germs) of $(z^{N-1}N^{-s})_{N\geq 2}$,
which is given in Example \ref{g-r=1}.
By the induction hypothesis, 
for any $0\leq k\leq k_0-2$, the sequence (of germs) of $(\mathrm{Li}^{\star}_{(z)}(s)_{\geq N})_{N\geq 2}$ at $a+k+1$ 
has asymptotic expansion to precision $N^{-A}$. Hence for each $0\leq k\leq k_0-2$, the sequence of germs 
$(\mathrm{Li}^{\star}_{(z)}(s+k+1)_{\geq N})_{N\geq 2}$ at $a$ has asymptotic expansion to precision $N^{-A}$ 
and therefore the sequence (of germs) of $(\mathrm{Li}^{\star}_{(z)}(s)_{\geq N})_{N\geq 2}$ at $a$ has an asymptotic 
expansion to precision $N^{-A}$.

To write down the corresponding formal asymptotic expansion, we again need an inversion process.
From the above arguments we have that for each $0\leq j\leq k_0-1$, 
the sequence (of germs) of meromorphic functions $(z^{N-1}N^{-s-j})_{N\geq 2}$ at $a$ has the same formal
asymptotic expansion to precision $N^{-A}$ as that of the sequence (of germs) of meromorphic functions
$$
\left(\left({\bar z}-1\right)\mathrm{Li}^{\star}_{(z)}(s+j)_{\geq N}+\sum_{k= 0}^{k_0-j-2}(-1)^k\frac{(s+j)_{k+1}}{(k+1)!}\mathrm{Li}^{\star}_{(z)}(s+j+k+1)_{\geq N}\right)_{N\geq 2}
$$
at $a$. Writing this in matrix from, we have
\begin{equation}\label{matrix}
\mathbf{W=AU},
\end{equation}
where $\bf{U}$ and $\bf{W}$ are column vectors, whose entries are the formal asymptotic expansion 
at $a$ to precision $N^{-A}$ of the column vectors
\begin{equation}\label{Eulerian-matrix}
\begin{bmatrix}
\mathrm{Li}^{\star}_{(z)}(s)_{\geq N}   \\
 \mathrm{Li}^{\star}_{(z)}(s+1)_{\geq N}\\
 \vdots\\
\mathrm{Li}^{\star}_{(z)}(s+k_0-1)_{\geq N}
\end{bmatrix}
\text{and} 
 \begin{bmatrix}
 z^{N-1}N^{-s}\\
 z^{N-1}N^{-s-1}\\
 \vdots\\
 z^{N-1}N^{-s-k_0+1}
\end{bmatrix},
\end{equation}
respectively, and $\mathbf{A}$ is the square matrix whose entries are the germs at $a$ of
the rational functions of $s$ as per the matrix given below:
$$
\mathbf{A}(s)=\begin{bmatrix}
{\bar z}-1 &\frac{(s)_1}{1!} &\frac{(-1)^1(s)_2}{2!} &\cdots &\frac{(-1)^{k_0-2}(s)_{k_0-1}}{(k_0-1)!}\\ 
     0                         &{\bar z}-1          &\frac{(s+1)_1}{1!}       &\cdots  &\frac{(-1)^{k_0-3}(s+1)_{k_0-2}}{(k_0-2)!}\\
     \vdots                 &\vdots                              &\vdots                           &\ddots         &\vdots\\
     0                            &0                                      &0                           &\cdots     &{\bar z}-1
\end{bmatrix}.
$$
Note that as $z\neq 1$, the matrix $\mathbf{A}(s)$ is invertible in $\mathbf{T}(\C(s))$, the ring of upper 
triangular matrices of order $k_0\times k_0$ with elements in the field of rational functions $\C(s)$. 
To find the inverse, we use the idea from \cite[\S 5.2]{BS1} with suitable changes. Note that
$$
\mathbf{A}(s)={\bar z}\mathbf{I}-\exp(-\mathbf{M}(s)),
$$
where $\mathbf{I}$ denotes the identity matrix of order $k_0\times k_0$ and $\mathbf{M}(s)$ is the 
nilpotent matrix over $\C(s)$ defined by
$$
\mathbf{M}(s)=\begin{bmatrix}
0  &s &0 &\cdots &0 \\ 
     0                         &0          &s+1      &\cdots  &0 \\
     \vdots                 &\vdots                              &\vdots                           &\ddots         &\vdots\\
     0                            &0                                      &0                          &\cdots                &0
\end{bmatrix}.
$$
Now we need to invert the power series $c-e^{-x}$ in $\C{[[x]]}$ for some $c\neq 1$. For this, 
we recall the definition of Eulerian polynomials. The Eulerian polynomials $A_n(t)$ for $t\neq 1$, 
are defined by the following generating series:
$$
\frac{1-t}{e^{(t-1)y}-t}=\sum_{n\geq 0}A_n(t)\frac{y^n}{n!}.
$$
Hence
$$
\frac{1}{c-e^{-x}}=\frac{1}{c-1}\sum_{n\geq 0}(-1)^n\frac{A_n(c)}{(c-1)^n}\frac{x^n}{n!}.
$$
It can be checked that $A_0(t)=1$. Now define $A^{\star}_{n}(c):=(-1)^{n}A_n(c)$ for all complex
numbers $c\neq 1$ and integers $n\geq 1$. Then the inverse of ${\bf A}(s)$ in $\mathbf{T}(\C(s))$ is given by
\begin{align*}
{\bf A}(s)^{-1}&=\frac{1}{{\bar z}\mathbf{I}-\exp(-\mathbf{M}(s))}=
\left({\bar z}-1\right)^{-1}\sum_{n=0}^{k_0-1}\frac{A^{\star}_{k}({\bar z})}{({\bar z}-1)^n}\frac{(\mathbf{M}(s))^n}{n!},
\end{align*}
i.e.,
\begin{align}\label{A-inverse}
\mathbf{A}(s)^{-1}=\left({\bar z}-1\right)^{-1}\begin{bmatrix}
1 &\frac{(s)_1A_1^{\star}({\bar z})}{({\bar z}-1)1!} &\frac{(s)_2A_2^{\star}({\bar z})}{({\bar z}-1)^22!} &\cdots &\frac{(s)_{k_0-1}A_{k_0-1}^{\star}({\bar z})}{({\bar z}-1)^{k_0-1}(k_0-1)!}\\ 
     0                         &1          &\frac{(s+1)_1A_1^{\star}({\bar z})}{({\bar z}-1)1!}       &\cdots  &\frac{(s+1)_{k_0-2}A_{k_0-2}^{\star}({\bar z})}{({\bar z}-1)^{k_0-2}(k_0-2)!}\\
     \vdots                 &\vdots                              &\vdots                           &\ddots         &\vdots\\
     0                            &0                                      &0                           &\cdots     &1
\end{bmatrix}.
\end{align}
Thus ${\bf A}$ is invertible and from \eqref{matrix} we have
$$
\mathbf{U}=\mathbf{A}^{-1}\mathbf{W}.
$$
By comparing the first entries on both the sides, we get that the sequence 
(of germs at $a$) of $(\mathrm{Li}^{\star}_{(z)}(s)_{\geq N})_{N\geq 2}$ has the same formal asymptotic 
expansion to precision $N^{-A}$ as that of
$$
\left(\sum_{k=0}^{k_0-1}\frac{(s)_{k}A_{k}^{\star}({\bar z})}{({\bar z}-1)^{k+1}k!}z^{N-1}N^{-s-k}\right)_{N\geq 2}.
$$
Since $z\neq 1$, therefore by \exref{g-r=1}, this formal asymptotic expansion to precision $N^{-A}$
(relative to $\mathcal{E}$ with coefficients in $\mathcal{C}_{(z)}$)
is zero. This implies that the formal asymptotic 
expansion to precision $n^{-A}$ (of germs at $a$) of $(\mathrm{Li}^{\star}_{(z)}(s)_{\geq N})_{N\geq 2}$
(relative to $\mathcal{E}$ with coefficients in $\mathcal{C}_{(z)}$)
is also equal to zero.  Since the positive integer $A$ is arbitrary, this completes the proof.
\end{proof}

\subsection{Asymptotic expansion of $\mathrm{Li}^{\star}_{\bf z}(s_1,\ldots,s_r)_{\geq N}$}\label{asymp-li-star}
 Let $r\geq1$ be an integer. Recall that for an integer $N\geq 2$ and $(s_1,\ldots,s_r)\in U_r$, 
$\mathrm{Li}^{\star}_{\bf z}(s_1,\ldots,s_r)_{\geq N}$ is defined by the series
\begin{align*}
  \sum_{n_1\geq \cdots \geq n_r \geq N}\frac{z_1^{n_1}\cdots z_r^{n_r}}{n_1^{s_1}\cdots n_r^{s_r}}.
\end{align*}
It is a holomorphic function on $U_r$ and can be extended to a meromorphic function on $\C^r$.
We need to define some notations to state our next proposition. For each $1\leq i\leq r$,
define $q_{[1,i]}:=$1 if  $z_{[1,i]}=1$, else put $q_{[1,i]}=0$. Consider the integer $Q_{[1,i]}$ 
which counts the number of indices $1\leq j\leq i$ such that $z_{[1,j]}=1$, i.e.,
 $$
 Q_{[1,i]}:=q_{[1,1]}+q_{[1,2]}+\cdots+q_{[1,i]}.
 $$
 Let the sets $I_{i}$ and $I'_{i}$ be defined as follows:
 $$
 I_{i}=\{i: 1\leq j\leq i \text{ and }  z_{[1,j]}=1\}\\
 \text{ and }\\
  I'_{i}=\{j: 1\leq j\leq i \text{ and }  z_{[1,j]}\neq 1\}.
 $$
Clearly $Q_{[1,i]}=|I_{i}|$ and $I_{i}\cup I'_{i}=\{1,\ldots,i\}.$ Also, for $\mathbf{s}=(s_1,\ldots,s_r)\in \C^r$, $|\mathbf{s}|$ 
denotes the sum of its coordinates, i.e., $|\mathbf{s}|=s_1+\cdots+s_r$. Then, we prove the following proposition which 
generalises \propref{*-mp-asym-r=1} and \cite[Remark 17]{BS2} for $r\geq 2$.

 \begin{prop}\label{asymp-mp*}
 Let $r\geq 1$ be an integer and $\mathbf{a}=(a_1,\ldots,a_r)\in \Z^r$. The sequence (of germs at $\mathbf{a}$) of the 
 meromorphic functions $(\mathrm{Li}^{\star}_{\bf z}(s_1,\ldots,s_r)_{\geq N})_{N\geq 2}$ has a complete 
 asymptotic expansion relative to $\mathcal{E}$ with coefficients in $\mathcal{C}_{\bf z}$. Moreover, the associated formal 
 complete asymptotic expansion is the formal Laurent series
 $$
 \sum_{\mathbf{k}\in \N^r}\sum_{l\geq 0}h^{\star}_{(l,\mathbf{k})}L^{l}X^{|\mathbf{a}|+|\mathbf{k}|-Q_{[1,r]}},
 $$
 where for $\mathbf{k}=(k_1,\ldots,k_r)\in \N^r$, $h^{\star}_{(l,\mathbf{k})}$ is the germ at $\mathbf{a}$ of the
 holomorphic function defined by
 \begin{align*}
  \mathbf{s}=(s_1,\ldots,s_r)\mapsto & \ \frac{q_{[1,r]} (-1)^l}{l!k_1!\cdots k_r! }
  \prod_{j\in I'_{r}}
  \frac{{\overline{z_{[1,j]}}}A^{\star}_{k_j}({\overline{z_{[1,j]}}})}{({\overline{z_{[1,j]}}}-1)^{k_j+1}}\prod_{j\in I_{r}}B^{\star}_{k_j}\\
  &\times(s_1)_{k_1-q_{[1,1]}} (s_1+s_2+k_1-Q_{[1,1]})_{k_2-q_{[1,2]}}\cdots\\
  &\times(s_1+\cdots+s_r+k_1+\cdots+k_{r-1}-Q_{[1,r-1]})_{k_r-q_{[1,r]}} (|\mathbf{s}|-|\mathbf{a}|)^l.
\end{align*}
 \end{prop}

\begin{proof}
To prove the proposition, we will use induction on $r$. Let $r=1$. If $z=1$,
then we have the desired result by \cite[Remark 17]{BS2}. If $z\neq 1$ 
then, $q_{[1,1]}=0$ and hence by \propref{*-mp-asym-r=1}, $h^{\star}_{(l,k)}=0$ for all $(l,k)\in \N\times\N$.
Therefore, we have the desired result.

Next assume that $r\geq 2$. Let $A$ be a positive integer and $k_0$ be the smallest non-negative integer such that
$(a_1+k_0-A,a_2,\ldots,a_r)\in U_r$. Let $D$ be a polydisc around ${\bf a}$ with small polyradius $(\rho_1,\ldots,\rho_r)$
such that $\rho_1+\cdots+\rho_r<1$. If $k_0=0 $, then $(a_1-A,a_2,\ldots,a_r)\in U_r$.  Therefore,
$\lVert\mathrm{Li}^{\star}_{\bf z}(s_1,\ldots,s_r)_{\geq N}\rVert_D=o(N^{-A})$ as 
$N \to \infty$. Hence by  \exref{zero-asymp}, the sequence (of germs  at ${\bf a}$) of the holomorphic functions
$(\mathrm{Li}^{\star}_{\bf z}(s_1,\ldots,s_r)_{\geq N})_{N\geq 2}$
 has an asymptotic expansion to precision $N^{-A}$
 (relative to $\mathcal{E}$ and coefficients in $\mathcal{C}_{\bf z}$).
    
Now assume that $k_0\geq 1$. We will use induction on $k_0$. First, consider the case when $z_1=1$. Then from \eqref{eq2} we have,
\begin{align*}
&\mathrm{Li}^{\star}_{(z_{[1,2]},z_{3},\ldots,z_r)}(s_1+s_{2}-1,s_{3},\ldots,s_r)_{\geq N}
=\sum_{k\geq0}(-1)^k\frac{(s_1-1)_{k+1}}{(k+1)!}\mathrm{Li}^{\star}_{\bf z}(s_1+k,s_2,\ldots,s_r)_{\geq N}.
\end{align*}
Applying \lemref{uni-convg}, $D$ can be so chosen that as $N \to \infty$, we have
$$
\sum_{k\geq k_0}\left\lVert (-1)^k\frac{(s_1-1)_{k+1}}{(k+1)!}\mathrm{Li}^{\star}_{\bf z}(s_1+k,s_2,\ldots,s_r)_{\geq N} \right\rVert_D=o(N^{-A}).
$$
Therefore, by the induction hypothesis on the depth and \rmkref{change-variable}, the sequence (of germs at $\mathbf{a}$) 
of the meromorphic functions $(\mathrm{Li}^{\star}_{(z_{[1,2]},z_{3}\ldots,z_r)}(s_1+s_{2}-1,s_{3},\ldots,s_r)_{\geq N})_{N\geq 2}$ has an 
asymptotic expansion to precision $N^{-A}$ (relative to $\mathcal{E}$ and coefficients in $\mathcal{C}_{\bf z}$). Hence, this
together with \exref{zero-asymp} yield that the sequence (of germs at $\mathbf{a}$) of the meromorphic functions 
$$
\left(\sum_{k=0}^{k_0-1} (-1)^k\frac{(s_1-1)_{k+1}}{(k+1)!}
\mathrm{Li}^{\star}_{\bf z}(s_1+k,s_2,\ldots,s_r)_{\geq N}\right)_{N\geq 2}
$$
has an asymptotic expansion to precision $N^{-A}$ (relative to $\mathcal{E}$ and coefficients in $\mathcal{C}_{\bf z}$).
Note that for $1\leq k\leq k_0-1$, the sequence (of germs at 
$(a_1+k,a_2,\ldots,a_r)$) of $(\mathrm{Li}^{\star}_{\bf z}(s_1,\ldots,s_r)_{\geq N})_{N\geq 2}$
has an asymptotic expansion to precision $N^{-A}$, by the induction hypothesis on $k< k_0$. Thus for
$1\leq k\leq k_0-1$, the sequence (of germs at $\mathbf{a}$) of meromorphic functions 
$(\mathrm{Li}^{\star}_{\bf z}(s_1+k,s_2,\ldots,s_r)_{\geq N})_{N\geq 2}$ has an asymptotic 
expansion to precision $N^{-A}$. This implies that the sequence (of germs at 
$\mathbf{a}$) of the meromorphic functions $(\mathrm{Li}^{\star}_{\bf z}(s_1,\ldots,s_r)_{\geq N})_{N\geq 2}$
has an asymptotic expansion to precision $N^{-A}$  (relative to $\mathcal{E}$ and coefficients in $\mathcal{C}_{\bf z}$).

Now to write down the corresponding formal asymptotic expansion to precision $N^{-A}$,
we note that for $0\leq j\leq k_0-1$, the sequence (of germs at $\mathbf{a}$) of the meromorphic functions 
$\left(\sum_{k=0}^{k_0-j-1}(-1)^k \frac{(s_1+j-1)_{k+1}}{(k+1)!}\mathrm{Li}^{\star}_{\bf z}(s_1+j+k,s_2,\ldots,s_r)_{\geq N}\right)_{N\geq 2}$, 
and $(\mathrm{Li}^{\star}_{(z_{[1,2]},z_3,\ldots,z_r)}(s_1+s_2+j-1,s_3,\ldots,s_r)_{\geq N})_{N\geq 2}$
have the same formal asymptotic expansion to precision $N^{-A}$, by the above arguments.
This can be written as the following matrix identity:
 \begin{equation}\label{AX}
 \mathbf{Y}=\mathbf{BX},
 \end{equation}
where $\mathbf{X}, \mathbf{Y}$ are column vectors whose entries are the formal asymptotic expansion at 
$\mathbf{a}$ to precision $N^{-A}$ of the column vectors
$$
 \begin{bmatrix}
 \mathrm{Li}^{\star}_{\bf z}(s_1,\ldots,s_r)_{\geq N}     \\
 \mathrm{Li}^{\star}_{\bf z}(s_1+1,\ldots,s_r)_{\geq N}\\
 \vdots\\
 \mathrm{Li}^{\star}_{\bf z}(s_1+k_0-1,\ldots,s_r)_{\geq N}
\end{bmatrix}
\ \text{and} \
\begin{bmatrix}
 \mathrm{Li}^{\star}_{(z_{[1,2]},z_3,\ldots,z_r)}(s_1+s_2-1,s_3,\ldots,s_r)_{\geq N}    \\
 \mathrm{Li}^{\star}_{(z_{[1,2]},z_3,\ldots,z_r)}(s_1+s_2,s_3,\ldots,s_r)_{\geq N}\\
 \vdots\\
 \mathrm{Li}^{\star}_{(z_{[1,2]},z_3,\ldots,z_r)}(s_1+s_2+k_0-2,s_3,\ldots,s_r)_{\geq N}
\end{bmatrix},
$$
respectively, and $\mathbf{B}$ is the square matrix, whose entries are the germs at $\mathbf{a}$ of the rational
functions of $s_1$ as per the matrix given below:
$$
{\bf B}(s_1)=
\begin{bmatrix}
\frac{(s_1-1)_1}{1!} &\frac{(-1)(s_1-1)_2}{2!} &\frac{(-1)^2(s_1-1)_3}{3!} &\cdots &\frac{(-1)^{k_0-1}(s_1-1)_{k_0}}{k_0!}\\ 
     0                         &\frac{(s_1)_1}{1!}          &\frac{(-1)(s_1)_2}{2!}       &\cdots  &\frac{(-1)^{k_0-2}(s_1)_{k_0-1}}{(k_0-1)!}\\
     \vdots                 &\vdots                              &\vdots                           &\ddots         &\vdots\\
     0                            &0                                      &0                           &\cdots     &\frac{(s_1+k_0-2)_1}{1!}
\end{bmatrix}.
$$
The matrix $\mathbf{B}(s_1)$ is invertible in the ring of upper triangular matrices with entries in the field of rational functions 
$\Q(s_1)$, denoted by $\mathbf{T}(\Q(s_1))$ and is given by the following square matrix (see \cite[Remark 17]{BS2}):
$$
{\bf B}(s_1)^{-1}=
\begin{bmatrix}
\frac{1}{s_1-1} &\frac{B_1^{\star}}{1!} &\frac{(s_1)_1B_2^{\star}}{2!} &\cdots &\frac{(s_1)_{k_0-2}B_{k_0-1}^{\star}}{(k_0-1)!}\\ 
     0                         &\frac{1}{s_1}          &\frac{B_1^{\star}}{1!}       &\cdots  &\frac{(s_1+1)_{k_0-3}B_{k_0-2}^{\star}}{(k_0-2)!}\\
     \vdots                 &\vdots                              &\vdots                           &\ddots         &\vdots\\
     0                            &0                                      &0                           &\cdots     &\frac{1}{s_1+k_0-2}
\end{bmatrix},
$$
where for integers $n\ge 1$, $B_n^{\star}$ denote the star Bernoulli numbers, as before.
Hence the matrix ${\bf B}$ is invertible and we can rewrite equation \eqref{AX} as
$$
\mathbf{X}=\mathbf{B}^{-1}{\bf Y}.
$$
Comparing the first entries on both the sides, we get that the formal asymptotic expansion (of germs at $\mathbf{a}$) of
$(\mathrm{Li}^{\star}_{\bf z}(s_1,\ldots,s_r)_{\geq N})_{N\ge 2}$  
to precision $N^{-A}$ is same as that of 
$$
\left(\sum_{k=0}^{k_0-1}\frac{B_k^{\star}}{k!}(s_1)_{k-1}
\mathrm{Li}^{\star}_{(z_{[1,2]},z_3,\ldots,z_r)}(s_1+s_2+k-1,s_3,\ldots,s_r)_{\geq N} \right).
$$
Now by the induction hypothesis on depth $r$, we get that for $k\ge 0$, the formal complete asymptotic expansion of the 
sequence (of germs at $\mathbf{a}$) of
$(\mathrm{Li}^{\star}_{(z_{[1,2]},z_3,\ldots,z_r)}(s_1+s_2+k-1,s_3,\ldots,s_r)_{\geq N})_{N\ge 2}$ is 
$$
\sum_{k_2,\ldots,k_r\ge 0}\sum_{l\geq 0}h^{\star}_{(l,k,k_2,\ldots,k_r)}L^{l}
X^{|\mathbf{a}|+k+k_2+\cdots+k_r-1-(q_{[1,2]}+\cdots+q_{[1,r]})},
 $$
where $h^{\star}_{(l,k,k_2,\ldots,k_r)}$ is the germs at $\mathbf{a}$ of the function
\begin{align*}
\mathbf{s}\mapsto &\ \frac{q_{[1,r]} (-1)^l}{l!k_2!\cdots k_r! }
  \prod_{j\in J{'}}
  \frac{{\overline{z_{[1,j]}}}A^{\star}_{k_j}({\overline{z_{[1,j]}}})}{({\overline{z_{[1,j]}}}-1)^{k_j+1}}\prod_{j\in J}B^{\star}_{k_j}\\
 &\times (s_1+s_2+k-1)_{k_2-q_{[1,2]}}(s_1+s_2+s_3+k+k_2-1-q_{[1,2]})_{k_3-q_{[1,3]}}\cdots\\
 &\times (s_1+\cdots+s_r+k+k_2+\cdots+k_{r-1}-1-q_{[1,2]}-\cdots-q_{[1,r-1]})_{k_r-q_{[1,r]}} (|\mathbf{s}|-|\mathbf{a}|)^l,
\end{align*}
with $J$ being the set of integers $2\leq j\leq r$ such that $z_{[1,j]}=1$
and $J'$ being the set of integers $2\leq j\leq r$ such that $z_{[1,j]}\neq1$.
Note that $z_1=1$, implies that $q_{[1,1]}=1$. Hence $I_{r}=J\cup\{1\}$ and $I'_{r}=J'$. Hence from the
 above argument we get that the formal complete asymptotic expansion of the sequence (of germs at $\mathbf{a}$)
 of the meromorphic function $(\mathrm{Li}^{\star}_{\bf z}(s_1,\ldots,s_r)_{\geq N})_{N\ge 2}$ is given by 
$$
\sum_{\mathbf{k}\in \N^r}\sum_{l\geq 0}h^{\star}_{(l,\mathbf{k})}L^{l}X^{|\mathbf{a}|+|\mathbf{k}|-Q_{[1,r]}},
$$
as $Q_{[1,r]}=1+q_{[1,2]}+\cdots+q_{[1,r]}=q_{[1,1]}+\cdots+q_{[1,r]}$. So, we have desired result when $z_1=1$.

Now suppose $z_1\neq 1$. In this case, as before we write \eqref{eq2} as
 \begin{align*}
 \begin{split}
&{\overline{z_1}}\mathrm{Li}^{\star}_{(z_{[1,2]},z_{3},\ldots,z_r)}(s_1+s_{2},s_{3},\ldots,s_r)_{\geq N}\\
&=\left({\overline{z_1}}-1\right)\mathrm{Li}^{\star}_{\bf z}(s_1,\ldots,s_r)_{\geq N}+
\sum_{k\geq0}(-1)^k\frac{(s_1)_{k+1}}{(k+1)!}\mathrm{Li}^{\star}_{\bf z}(s_1+k+1,s_2,\ldots,s_r)_{\geq N}.
\end{split}
\end{align*}
Again similar to the previous case, using the \lemref{uni-convg}, induction on depth $r$, \rmkref{change-variable} 
and then induction on $k_0$, which is the smallest non-negative integer such that $(a_1+k_0-A,a_2,\ldots,a_r)\in U_r$, 
we get that the sequence (of germs at ${\bf a}$) of the meromorphic functions 
$$
\left(\mathrm{Li}^{\star}_{\bf z}(s_1,\ldots,s_r)_{\geq N}\right)_{N\geq 2}
$$
has an asymptotic expansion to precision $n^{-A}$ (relative to $\mathcal{E}$ and coefficients in 
$\mathcal{C}_{\bf z}$). Moreover, the sequence (of germs at ${\bf a}$) of 
$\left({\overline{z_1}}\mathrm{Li}^{\star}_{(z_{[1,2]},z_{2},\ldots,z_r)}(s_1+s_2,s_{3},\ldots,s_r)_{\geq N}\right)_{N\geq 2}$ 
and 
$$
\left(({\overline{z_1}}-1)\mathrm{Li}^{\star}_{\bf z}(s_1,\ldots,s_r)_{\geq N}+\sum_{k=0}^{k_0-2}
\frac{(-1)^k(s_1)_{k+1}}{(k+1)!}\mathrm{Li}^{\star}_{\bf z}(s_1+k+1,s_2,\ldots,s_r)_{\geq N}\right)_{N\geq 2}
$$
have the same formal asymptotic expansion to precision $n^{-A}$ (relative to $\mathcal{E}$ 
and coefficients in $\mathcal{C}_{\bf z}$).
Now we write down this expansion. As before, the above statement is true if we replace
$(s_1,s_{2},\ldots,s_r)$ by $(s_1+j,s_{2},\ldots,s_r)$ and $k_0$ by $k_0-j$ for each $0\leq j\leq k_0-1$.
Again, we write this as the following matrix identity:
 \begin{equation}\label{AX2}
 {\overline{z_1}}\mathbf{Z}=\mathbf{AX},
 \end{equation}
where $\mathbf{X}$ is as before, ${\bf Z}$ is obtained from $\mathbf{Y}$ by replacing $s_1$ with $s_1+1$,
and $\mathbf{A}$ is the square matrix as in \eqref{matrix} (with $s, z$ replaced by 
$s_1$ and $z_1$, respectively). Therefore, we can write
$$
\mathbf{X}= {\overline{z_1}}\mathbf{A}^{-1}{\bf Z}.
$$
Now comparing the first entries on both the sides, we get that the formal asymptotic expansion (of germs at ${\bf a}$)
of $(\mathrm{Li}^{\star}_{\bf z}(s_1,\ldots,s_r)_{\geq N})_{N\ge 2}$ to precision $N^{-A}$ is the same as that of 
$$
\left(\sum_{k=0}^{k_0-1}\frac{\overline{z_1}A_k^{\star}({\overline{z_1}})}{({\overline{z_1}}-1)^{k+1}k!}(s_1)_{k} \mathrm{Li}^{\star}_{(z_{[1,2]},z_3,\ldots,z_r)}(s_1+s_2+k,s_3,\ldots,s_r)_{\geq N} \right)_{N\geq 1}.
$$
Note that $q_{[1,1]}=0$ in this case. Hence the proof follows from the induction hypothesis on depth,
as in the case of $z_1=1$.
 \end{proof}
 
\section{Behaviour of the multiple polylogarithms at general integer points}\label{sec-gen}

As before, let $r \ge 1$ be an integer and
let ${\bf z}=(z_1,\ldots,z_r) \in \C^r$ be such that $z_1,\ldots,z_r$ are fixed roots of unity.
In \S\ref{behavior-at-boundary}, we have studied the local behaviour of
the multiple polylogarithms at integer points of  $\overline{U_r(\mathbf{z})}$. In this section, our aim is to study the
local behaviour of the multiple 
polylogarithms at general integer points of $\C^r$. 
To state our theorem we first need some notations.

For each $1\leq j\leq i\leq r$, define $q_{[j,i]}=$1 if  $z_{[j,i]}=1$, otherwise put $q_{[j,i]}=0$.
Consider the integer $Q_{[j,i]}$ which counts the integers $k$ such $j\leq k\leq i$ and $z_{[k,i]}=1$, i.e.,
 $$
 Q_{[j,i]}:=q_{[j,i]}+q_{[j+1,i]}+\cdots+q_{[i,i]}.
 $$
 Let $J_{i}$ be the set of integers $1\leq j\leq i$ (as in \S\ref{sec-intro}) such that $z_{[j,i]}=1$ and $J'_{i}$
 be the set of integers $1\leq j\leq i$ such that $z_{[j,i]}\neq 1$, i.e.,
 $$
 J_{i}=\{j: 1\leq j\leq i \text{ and } z_{[j,i]}=1\}\\
 \text{ and }\\
  J'_{i}=\{j: 1\leq j\leq i \text{ and }  z_{[j,i]}\neq 1\}.
 $$
 These notations are set in the reverse ordering when compared to the notations in \S \ref{asymp-li-star}, as we need
 to use the combinatorial identity \eqref{combi-indenty} where the  indices in the tail of the multiple
 polylogarithm-star function are in the reverse order.
In this context, we prove the following theorem.
\begin{thm}\label{thm-gen-point}
Given ${\bf a}=(a_1,\ldots,a_r)\in \Z^r$, the formal power series
\begin{align*}
\sum_{k_1,\ldots,k_r\geq 0}\frac{(-1)^{k_1+\cdots+k_r}}{k_1!\cdots k_r!}\ell_{[k_1,\ldots,k_r]}^{(a_1,\ldots,a_r)}(\mathbf{z})(s_1-a_1)^{k_1}\cdots(s_r-a_r)^{k_r}
\end{align*}
converges in a neighbourhood of ${\bf a}$ and extends to a meromorphic
function on $\C^r$. If we denote it by $\mathrm{Li}^{{\rm Reg}}_{(\mathbf{z};\mathbf{a})}(s_1,\ldots,s_r)$,
then we have the following equality between meromorphic functions
\begin{align}\label{behaviour-gen-point}
\mathrm{Li}^{{\rm Reg}}_{(\mathbf{z};\mathbf{a})}(s_1,\ldots,s_r)
=\sum_{i=0}^{r}(-1)^i \mathrm{Li}_{(z_{i+1},\ldots,z_r)}(s_{i+1},\ldots,s_r) C_i(s_1,\ldots,s_i),
\end{align}
where $C_0(\varnothing):=1$ and for $1 \le i \le r$,
\begin{align*}
\begin{split}
C_i(s_1,\ldots,s_i):=q_{[1,i]}\sum_{\substack{k_1,\ldots,k_i\geq 0 \\ \sum_{1\leq t\leq i}(k_t+a_t)=Q_{[1,i]}}}
&\frac{1}{k_1!\cdots k_i!}
\prod_{j\in J_{i}^{'}}\frac{\overline{z_{[j,i]}}A_{k_j}^{\star}(\overline{z_{[j,i]}})}{(\overline{z_{[j,i]}}-1)^{k_j+1}} \prod_{j\in J_{i}}B_{k_j}^{\star}\\
& \times (s_i)_{k_i-q_{[i,i]}}(s_i+s_{i-1}+k_i-Q_{[i,i]})_{k_{i-1}-q_{[i-1,i]}}\cdots\\
& \times (s_i+\cdots+ s_{1}+k_i+\cdots+k_2-Q_{[2,i]})_{k_{1}-q_{[1,i]}}.
\end{split}
\end{align*}
\end{thm}

As a key corollary, we first prove the holomorphicity of the multiple polylogarithm function at integer
points of $V_r(\mathbf{z})$. 

\begin{cor}\label{cor-vrz}
Given $\mathbf{a}\in \Z^r\cap V_r(\mathbf{z})$, the function $\mathrm{Li}_{\bf z}(s_{1},\ldots,s_r)$ is 
holomorphic at $\mathbf{a}$ and its value at $\mathbf{a}$ is equal to the regularised value
$\ell_{[0,\ldots,0]}^{(a_1,\ldots,a_r)}(\mathbf{z})$ of the series \eqref{mp} at $\mathbf{a}$.
\end{cor}

\begin{proof} Let $\mathbf{a}=(a_1,\ldots,a_r)$. Since $\mathbf{a}\in \Z^r\cap V_r(\mathbf{z})$, then for 
every $1\leq i \leq r$, we have $a_1+\cdots+a_i>Q_i({\bf z})=Q_{[1,i]}$. Hence, for every $1\leq i \leq r$,
there are no non-negative integers $k_1, \ldots, k_i$ such that $\sum_{1\leq t\leq i}(k_t+a_t)=Q_{[1,i]}$, i.e.,
for every $1\leq i \leq r$, $C_i(s_1,\ldots,s_i)=0$.
So by \eqref{behaviour-gen-point}, around the point $\mathbf{a}$, we have
$$
\mathrm{Li}_{\bf z}(s_{1},\ldots,s_r)=\mathrm{Li}^{\text{Reg}}_{(\mathbf{z};\mathbf{a})}(s_1,\ldots,s_r).
$$
 Hence
 $\mathrm{Li}_{\bf z}({\bf a})=\mathrm{Li}^{\text{Reg}}_{(\mathbf{z};\mathbf{a})}({\bf a})
 =\ell_{[0,\ldots,0]}^{(a_1,\ldots,a_r)}(\mathbf{z})$. 
\end{proof}

We recall that in \cite[Theorem 3]{PSBS}, we proved that for
$\mathbf{a}\in \Z^r\cap V_r(\mathbf{z})$, the series \eqref{mp}, treated as \eqref{lim-mp}, converges
at ${\bf a}$ to $\ell_{[0,\ldots,0]}^{(a_1,\ldots,a_r)}(\mathbf{z})$.
But such phenomenon can occur at certain points outside $V_r({\bf z})$ as well, even if the limit
in \eqref{lim-mp} does not exist. In view of \cite[Theorem 9]{BS1}, if $\mathbf{z}=(z_1,\ldots,z_r)\in \C^r$ be
such $z_i$'s are roots of unity and $z_{[1,i]}\neq 1$ for all $1\leq i\leq r$, then $\mathrm{Li}_{\bf z}(s_{1},\ldots,s_r)$
is holomorphic everywhere in $\C^r$. Moreover, we have the following corollary which is immediate from
\eqref{behaviour-gen-point}.

\begin{cor}
Let $\mathbf{z}=(z_1,\ldots,z_r)\in \C^r$ be such that $z_i$'s are roots of unity and $z_{[1,i]}\neq 1$ 
for all $1\leq i\leq r$. Then, for any $\mathbf{a}\in \Z^r$, the value of the holomorphic function
$\mathrm{Li}_{\bf z}(s_{1},\ldots,s_r)$ at $\mathbf{a}$ is equal to the regularised value
$\ell_{[0,\ldots,0]}^{(a_1,\ldots,a_r)}(\mathbf{z})$ of the series \eqref{mp} at $\mathbf{a}$.
\end{cor}

\begin{rmk}\rm
We also note that we can recover \eqref{behaviour-boundary} from \eqref{behaviour-gen-point}.
Let $(a_1,\ldots,a_r)\in \Z^r\cap \overline{U_r(\mathbf{z})}$. If $1\leq i\leq r$ is such that
$i \notin \mathcal{I}({\bf z},{\bf a})$, then $i$ does not satisfy one of the defining conditions of
$\mathcal{I}({\bf z},{\bf a})$. If $z_{[1,i]} \ne 1$, then $q_{[1,i]} =0$. Next if there exist
$1 \le i_1 < i_2 \le i$ such that $z_{[i_1,i]},z_{[i_2,i]} \ne 1$, then
$Q_{[1,i]}=q_{[1,i]}+\cdots+q_{[i,i]} \le i-2$. But since $a_1+\cdots+a_i \ge i-1$,
there are no non-negative integers $k_1, \ldots, k_i$ such that $\sum_{1\leq t\leq i}(k_t+a_t)=Q_{[1,i]}$.
Next let $a_1+\cdots+a_i > i-a({\bf z})_i$. Note that $a({\bf z})_i=0$ if $z_1=\cdots=z_i=1$, else
$a({\bf z})_i=1$. Also, $Q_{[1,i]}=i$ if $z_1=\cdots=z_i=1$, else $Q_{[1,i]} \le i-1$.
In any case, $Q_{[1,i]} < a_1+\cdots+a_i$ and hence
there are no non-negative integers $k_1, \ldots, k_i$ such that $\sum_{1\leq t\leq i}(k_t+a_t)=Q_{[1,i]}$.
So if $i \notin \mathcal{I}({\bf z},{\bf a})$, then $C_i(s_1,\ldots,s_i)=0$. 

Next let $i \in \mathcal{I}({\bf z},{\bf a})$ and $i \ge 1$. Note that similar computations as above
shows that the admissible choice of $k_1, \ldots, k_i$ in this case is $k_1= \cdots= k_i=0$.
Now if $t_i=0$, i.e., $z_1=\cdots=z_i=1$,
we have $J_{i}^{'}=\varnothing$ and $Q_{[j,i]}=i-j+1$ for $1 \le j \le i$. Hence
$$
(s_i+\cdots+ s_{j}-Q_{[j+1,i]})_{-q_{[j,i]}}
=\frac{1}{s_i+\cdots+ s_{j}-(i-j+1)}.
$$
So we are done in this case. Next let $t_i \ge 1$. In this case, we note that
$J_{i}^{'}=\{t_i\}$, $q_{[t_i,i]}=0$ and $q_{[j,i]}=1$ for all $1 \le j \le i$ such that $j \neq t_i$.
Thus, $Q_{[j,i]}=i-j+1$ for $t_i< j \le i$ and $Q_{[j,i]}=i-j$ for $1 \le j \le t_i$.
Now, as before, we note that for $t_i< j \le i$,
$$
(s_i+\cdots+ s_{j}-Q_{[j+1,i]})_{-q_{[j,i]}}
=\frac{1}{s_i+\cdots+ s_{j}-(i-j+1)},
$$
for $1 \le j < t_i$,
$$
(s_i+\cdots+ s_{j}-Q_{[j+1,i]})_{-q_{[j,i]}}
=\frac{1}{s_i+\cdots+ s_{j}-(i-j)},
$$
and for $j=t_i$,
$$
(s_i+\cdots+ s_{j}-Q_{[j+1,i]})_{-q_{[j,i]}}=1.
$$
Therefore, we see that  \eqref{behaviour-boundary} can be recovered from \eqref{behaviour-gen-point}.
\end{rmk}

\begin{rmk}\rm
Let $\mathcal{I}({\bf z})$ denote the set $\{0\} \cup \{ i : 1 \le i \le r \text{ and } q_{[1,i]}=1\}$.
We can write \eqref{behaviour-gen-point} for all the points of the form $(a_{i+1},\ldots,a_r)$ for $i\in\mathcal{I}({\bf z})$.
As in \S\ref{behavior-at-boundary}, all these equations can be concisely written
in the following matrix form:
$$
{\bf V}^{\text{Reg}}={\bf N}{\bf V},
$$
where ${\bf V}^{\text{Reg}}$ and ${\bf V}$ are column vectors indexed by the elements of 
$\mathcal{I}({\bf z})$, such that for each $i\in \mathcal{I}({\bf z})$, 
the corresponding entries of ${\bf V}^{\text{Reg}}$ and ${\bf V}$ are 
$\mathrm{Li}^{\text{Reg}}_{(z_{i+1},\ldots,z_r;a_{i+1},\ldots,a_r)}(s_{i+1},\ldots,s_r)$
and $\mathrm{Li}_{(z_{i+1},\ldots,z_r)}(s_{i+1},\ldots,s_r)$, respectively,
and ${\bf N}$ is an upper triangular matrix with entries in the field of rational functions
$\C(s_1,\ldots,s_r)$, and the diagonal entries as $1$. Hence, ${\bf N}$ is an invertible matrix with entries in
$\C(s_1,\ldots,s_r)$ and we can write
$$
{\bf V}={\bf N}^{-1}{\bf V}^{\text{Reg}}.
$$
Now the first entry of the above matrix equation gives an expression of the form
$$
\mathrm{Li}_{\bf z}(s_{1},\ldots,s_r)=\sum_{i \in \mathcal{I}({\bf z})} D_i(s_1,\ldots,s_i)\mathrm{Li}^{\text{Reg}}_{(z_{i+1},\ldots,z_r;a_{i+1},\ldots,a_r)}(s_{i+1},\ldots,s_r),
$$
where $D_0(\varnothing):=1$ and for $i \in \mathcal{I}({\bf z})$ with $i \ge 1$, $D_i(s_1,\ldots,s_i)$
is a rational function in $s_1,\ldots,s_i$. This equation, therefore, gives
the desired Laurent type expansion of $\mathrm{Li}_{\bf z}(s_{1},\ldots,s_r)$ around ${\bf a}$.
\end{rmk}

\begin{ex}\rm
Let ${\bf z}=(-1,1,-1,1)$ and ${\bf a}=(0,1,0,1)$. Then $\mathcal{I}({\bf z})=\{0,3,4\}$. From
\eqref{behaviour-gen-point}, we have
$$
\begin{bmatrix}
 \mathrm{Li}^{{\rm Reg}}_{(\mathbf{z};\mathbf{a})}(s_1,s_2,s_3,s_4)\\
 \mathrm{Li}^{\text{Reg}}_{(1;1)}(s_4)\\
1
\end{bmatrix}
= 
\begin{bmatrix}
1  &-\frac{1}{4(s_1+s_2+s_3-1)} &\frac{1}{4(s_4-1)(s_1+s_2+s_3+s_4-2)} \\ 
     0                         &1             &-\frac{1}{s_4-1}\\
         0                            &0                   &1
\end{bmatrix}
 \begin{bmatrix}
 \mathrm{Li}_{\bf z}(s_1,s_2,s_3,s_4)\\
 \mathrm{Li}_{(1)}(s_4)\\
1
\end{bmatrix}.
$$
Therefore,
$$ 
 \begin{bmatrix}
 \mathrm{Li}_{\bf z}(s_1,s_2,s_3,s_4)\\
 \mathrm{Li}_{(1)}(s_4)\\
1
\end{bmatrix}
=
\begin{bmatrix}
1  &\frac{1}{4(s_1+s_2+s_3-1)} &\frac{1}{4(s_1+s_2+s_3-1)(s_1+s_2+s_3+s_4-2)} \\ 
     0                         &1             &\frac{1}{s_4-1}\\
         0                            &0                   &1
\end{bmatrix}
\begin{bmatrix}
 \mathrm{Li}^{\text{Reg}}_{({\bf z}; {\bf a})}(s_1,s_2,s_3,s_4)\\
 \mathrm{Li}^{\text{Reg}}_{(1;1)}(s_4)\\
1
\end{bmatrix},
$$
and hence we get the following Laurent type expansion around the point ${\bf a}=(0,1,0,1)$:
\begin{align*}
\mathrm{Li}_{\bf z}(s_1,s_2,s_3,s_4)
= & \mathrm{Li}^{\text{Reg}}_{({\bf z}; {\bf a})}(s_1,s_2,s_3,s_4)
+\frac{ \mathrm{Li}^{\text{Reg}}_{(1;1)}(s_4)}{4(s_1+s_2+s_3-1)}\\
&+\frac{1}{4(s_1+s_2+s_3-1)(s_1+s_2+s_3+s_4-2)}.
\end{align*}
\end{ex}

\begin{proof}[Proof of \thmref{thm-gen-point}]
We first recall the combinatorial identity \eqref{combi-indenty},
\begin{align*}
\mathrm{Li}_{\bf z}(s_1,\ldots,s_r)_{< N}=\sum_{i=0}^{r}(-1)^i\mathrm{Li}^{\star}_{(z_i,\ldots,z_1)}(s_i,\ldots,s_1)_{\geq N}\mathrm{Li}_{(z_{i+1},\ldots,z_r)}(s_{i+1},\ldots,s_r).
\end{align*}
Note that $(\mathrm{Li}_{\bf z}(s_1,\ldots,s_r)_{< N})_{N \geq 2}$ is a sequence of holomorphic functions around 
${\bf a}$ and its sequence of the $\mathbf{k}=(k_1,\ldots,k_r)$-th Taylor coefficients at ${\bf a}$ is given by 
\begin{align}\label{taylor-coeffi}
\left(\frac{(-1)^{k_1+\cdots+k_r}}{k_1!\cdots k_r!}\sum_{N>n_1>\cdots>n_r>0}\frac{z_1^{n_1}(\log n_1)^{k_1}\cdots z_r^{n_r}(\log n_r)^{k_r}}{n_1^{a_1}\cdots n_r^{a_r}}\right)_{N\geq 2}.
\end{align}
Now using \propref{asymp-mp*}, \rmkref{change-variable} and combinatorial identity \eqref{combi-indenty}, we get that the sequence (of 
germs at ${\bf a}$) of holomorphic functions ($\mathrm{Li}_{\bf z}(s_1,\ldots,s_r)_{< N})_{N\geq 2}$  has a complete 
asymptotic expansion (relative to $\mathcal{E}$ and coefficient in $\mathcal{C}_{\bf z}$). Therefore, if the 
associated formal complete asymptotic expansion is given by $G=\sum_{(l,m)\in \N\times\Z}g_{(l,m)}L^lX^m$, then by 
definition for each $(l,m)\in \N\times\Z$, $g_{(l,m)}$ is the germ of a holomorphic function at ${\bf a}$. This implies that the 
sequence in \eqref{taylor-coeffi} has an asymptotic expansion to the arbitrary precision (relative to $\mathcal{E}$ and coefficient in 
$\mathcal{C}_{\bf z}$). In particular, by definition and the above argument, we get, for each $\mathbf{k}=(k_1,\ldots,k_r)$,
 $$
 c_{\mathbf{k}}(g_{(0,0)})=\frac{(-1)^{k_1+\cdots+k_r}}{k_1!\cdots k_r!}\ell_{[k_1,\ldots,k_r]}^{(a_1,\ldots,a_r)}(\mathbf{z}).
 $$ 
 Since $g_{(0,0)}$ is a germ of holomorphic function at ${\bf a}$, therefore the power series \eqref{Reg-Li}
 converges in a neighbourhood of ${\bf a}$.
 
 Now to get the required equality between the meromorphic functions, we apply \propref{asymp-mp*}
 and \rmkref{change-variable} in \eqref{combi-indenty}, and deduce that the formal Laurent series $G$ is
\begin{align*}
\sum_{i=0}^{r}\sum_{k_1,\ldots,k_r\geq 0}\sum_{l\geq 0}h_{(i,l,{\bf k})}(\mathbf{z})L^{l}X^{a_1+\cdots+a_i+k_1+\cdots+k_i-Q_{[1,i]}},
 \end{align*}
where for each $0\leq i\leq r$, $h_{(i,l,{\bf k})}(\mathbf{z})$ is the germ of the function at $\mathbf{a}$ which maps $(s_1,\ldots,s_r)$ to
\begin{align*}
& \ \frac{q_{[1,i]}(-1)^{i+l}}{l!k_1!\cdots k_i!}
\left(\prod_{j\in J_{i}^{'}}\frac{\overline{z_{[j,i]}}A_{k_j}^{\star}(\overline{z_{[j,i]}})}{(\overline{z_{[j,i]}}-1)^{k_j+1}}
\prod_{j\in J_{i}}B_{k_j}^{\star}\right)
\mathrm{Li}_{(z_{i+1},\ldots,z_r)}(s_{i+1},\ldots,s_r)(s_i)_{k_i-q_{[i,i]}}\\
&\times(s_i+s_{i-1}+k_i-Q_{[i,i]})_{k_{i-1}-q_{[i-1,i]}}\cdots(s_i+\cdots+ s_{1}+k_i+\cdots+k_2-Q_{[2,i]})_{k_{1}-q_{[1,i]}}
(|{\bf s}|-|{\bf a}|)^l.
\end{align*}
Since the power series \eqref{Reg-Li} is same as $g_{(0,0)}$,  taking $l=0$ and
$\sum_{1\leq t\leq i}(k_t+a_t)=Q_{[1,i]}$ for each $1\leq i\leq r$, we get that $g_{(0,0)}$ is the germ of function for which
\begin{align*}
(s_1,\ldots,s_r)\mapsto & \ \sum_{i=0}^{r}(-1)^{i}\mathrm{Li}_{(z_{i+1},\ldots,z_r)}(s_{i+1},\ldots,s_r) C_i(s_1,\ldots,s_i),
\end{align*}
where $C_i(s_1,\ldots,s_i)$'s are the desired rational functions. This completes the proof of the identity
\eqref{behaviour-gen-point} and the meromorphic extension of the power series \eqref{Reg-Li}.
\end{proof}

{\bf Statements and Declarations:}
The research of the first author is supported by PMRF (grant number: 1402688, cycle 10).
The authors have no competing interests to declare that are relevant to the content of this article.
This manuscript has no associated data.


\end{document}